\documentclass[11pt]{article}
\usepackage{amsmath,amsthm}
\usepackage{amssymb,mathrsfs}
\usepackage{bm,mathtools}
\usepackage{tikz}
\usetikzlibrary{calc}

\usepackage{xcolor}
\usepackage{geometry}
\usepackage[colorlinks=true,
linkcolor=blue!55!black,citecolor=blue!55!black,
urlcolor=blue!55!black,pagebackref]{hyperref}

\makeatletter
\def\@seccntDot{.}
\def\@seccntformat#1{\csname the#1\endcsname\@seccntDot\hskip 0.5em}
\renewcommand\section{\@startsection{section}{1}{\z@}%
{18\p@ \@plus 6\p@ \@minus 3\p@}%
{9\p@ \@plus 6\p@ \@minus 3\p@}%
{\large\bfseries\boldmath}}
\renewcommand\subsection{\@startsection{subsection}{2}{\z@}%
{15\p@ \@plus 6\p@ \@minus 3\p@}%
{6\p@ \@plus 6\p@ \@minus 3\p@}%
{\itshape}}
\renewcommand\subsubsection{\@startsection{subsubsection}{3}{\z@}%
{12\p@ \@plus 6\p@ \@minus 3\p@}%
{\p@}%
{}}
\makeatother

\usepackage{microtype}

\theoremstyle{plain}
\newtheorem{theorem}{Theorem}[section]
\newtheorem{lemma}{Lemma}[section]
\newtheorem{corollary}{Corollary}[section]
\newtheorem{proposition}{Proposition}[section]

\newtheorem{conjecture}{Conjecture}[section]

\theoremstyle{definition}

\newtheorem{remark}{Remark}[section]

\newtheorem{claim}{Claim}[section]

\numberwithin{equation}{section}
\allowdisplaybreaks
\DeclareMathOperator{\dist}{dist}

\DeclareMathOperator{\gap}{gap}

\title{Connected graphs with minimum adjacency spectral gap}
\author{
Lele Liu\footnote{School of Mathematical Sciences, Anhui University, Hefei 230601,
P.R. China. E-mail: \texttt{liu@ahu.edu.cn}. Supported by the National
Nature Science Foundation of China (No. 12471320), and Anhui Provincial Natural Science Foundation for Excellent Young Scholars (No. 2408085Y003).}~,~~~~
Michael Tait\footnote{Department of Mathematics \& Statistics, Villanova University, USA. E-mail: \texttt{michael.tait@villanova.edu}. Research partially supported by NSF grant DMS-2245556, a Villanova University Summer Grant, and a Villanova CLAS Research Semester.}~,~~~~
Yi Wang\footnote{School of Mathematical Sciences, Anhui University, Hefei 230601,
P.R. China. E-mail: \texttt{wangy@ahu.edu.cn}.
Supported by the National Natural Science Foundation of China (No. 12171002, 12331012)}
}
\date{}

\begin{document}
\maketitle

\begin{abstract}
Let $G$ be a connected graph, and let $\lambda_1(G) > \lambda_2(G)$ denote its two largest 
adjacency eigenvalues. The spectral gap of $G$ 
is defined as the difference  
$\lambda_1(G) - \lambda_2(G)$. For integers $r\geq 2$ and $s\geq 0$, the
double kite $DK(r,s)$ is formed by taking two vertex-disjoint copies of
the complete graph $K_r$ and joining one specified vertex of each clique to a path with $s$
internal vertices. Stani\'c (2013) conjectured that every
connected $n$-vertex graph with minimum adjacency spectral gap is a double
kite. In this paper, we confirm this conjecture for sufficiently large $n$.
\par\vspace{2mm}
\noindent{\bfseries Keywords:} adjacency eigenvalues; spectral gap; extremal graph; double kite
\par\vspace{2mm}
\noindent{\bfseries AMS Classification:} 05C35; 05C50; 15A18
\end{abstract}

\tableofcontents

\section{Introduction}

Consider a simple undirected graph $G$ on $n$ vertices, and let $A(G)$ denote its adjacency matrix. 
List its adjacency eigenvalues in nonincreasing order: $\lambda_1(G)\geq\lambda_2(G)\geq\cdots\geq\lambda_n(G)$.
The \emph{spectral gap} of $G$, denoted by $\gap(G)$, is defined as 
$\gap(G)=\lambda_1(G)-\lambda_2(G)$. If $G$ is connected, the Perron--Frobenius theorem implies
that $\lambda_1(G)$ is simple and has a positive eigenvector; in particular, $\gap(G) > 0$.  

The spectral gap is mainly investigated for regular graphs. 
For a $d$-regular graph $G$, we have $\lambda_1(G) = d$. Hence, $dI - A(G)$ is the
Laplacian and $I - A(G)/d$ is the normalized Laplacian of $G$.
Thus, in the regular setting the adjacency spectral gap is precisely the
algebraic connectivity introduced by Fiedler \cite{Fiedler1973}, and, after
division by $d$, it is also equivalent to spectral gap of the simple random walk.

The problem of minimizing the spectral gap within regular graphs has a long
history. Aldous and Fill \cite{Aldous-Fill} conjectured that the maximum
relaxation time of a random walk over all connected regular graphs on $n$
vertices is $(1 + o(1))\frac{3n^2}{2\pi^2}$. Equivalently, every connected $d$-regular graph $G$ on $n$ vertices should
satisfy $\gap(G)\geq (1+o(1))\frac{2d\pi^2}{3n^2}$,
with the bound asymptotically attained for at least one value of $d$.
For cubic graphs, Guiduli \cite{Guiduli1997} proved that a minimizer must be
path-like and built from specified blocks. Building on this, Brand, Guiduli, and Imrich
\cite{Brand2007} subsequently determined the unique cubic graph with minimum spectral gap, thereby
confirming a conjecture of Babai (see \cite{Guiduli1997}). Later, Abdi, Ghorbani, and Imrich
\cite{AbdiGhorbaniImrich2021} determined its asymptotic spectral gap is 
$(1+o(1))2\pi^2/n^2$. For quartic graphs, Abdi, Ghorbani, and Imrich
\cite{AbdiGhorbaniImrich2021} again obtained a path-like structural characterization
for the minimum spectral gap graphs. Subsequently, Abdi and Ghorbani
\cite{AbdiGhorbani2023} determined the structure of
connected quartic graphs on $n$ vertices with a minimum spectral gap, and
proved that the minimum gap is $(1+o(1))4\pi^2/n^2$. 

There is also a closely related extremal problem for the normalized
Laplacian without a regularity assumption. In this setting, spectral gap means smallest nonzero eigenvalue. Aksoy, Chung, Tait, and Tobin
\cite{AksoyChungTaitTobin2018} proved that the minimum normalized
Laplacian spectral gap over connected $n$-vertex graphs is $(1+o(1))54/n^3$, 
thereby settling another conjecture of Aldous and Fill \cite[p.\,216]{Aldous-Fill}.  

The purpose of this paper is to determine how small the
adjacency spectral gap can be among connected graphs of a given order and
to characterize the extremal graphs.
For integers $r\geq 2$ and $s\geq 0$, the \emph{double kite}
$DK(r,s)$ is formed from two vertex-disjoint copies of $K_r$ and a path on $s+2$ vertices 
by identifying each endpoint of the path with one vertex in each clique.
See Fig. \ref{fig:DK} for an illustration. 
The two identified vertices are called the \emph{attachment vertices}.

\begin{figure}[htbp]
\centering
\begin{tikzpicture}[main_node/.style={circle,draw,fill=blue!40,inner sep=.6mm}, scale=.9]
\foreach \i in {2,3,4,5,6,7}
  \node[main_node] (w\i) at (\i,0) {};

\draw[thick] (w7) -- (9,0);
\foreach \i in {0,1,2,...,7}
  \node[main_node] (v\i) at (45*\i:1.1) {};

\foreach \i in {0,1,2,...,7}
    \foreach \j in {0,1,2,...,7}
       \draw[thick] (v\i) -- (v\j);
%
\foreach \i in {0,1,2,...,7}
  \node[main_node] (u\i) at ($(9,0)+(45*\i:1.1)$) {};

\foreach \i in {0,1,2,...,7}
    \foreach \j in {0,1,2,...,7}
       \draw[thick] (u\i) -- (u\j);
\draw[thick] (v0) -- (w2) -- (w3) -- (w4) -- (w5) -- (w6) -- (w7);
\node[main_node] (u0) {};
\end{tikzpicture}
\caption{The double kite $DK(8,6)$}
\label{fig:DK}
\end{figure}
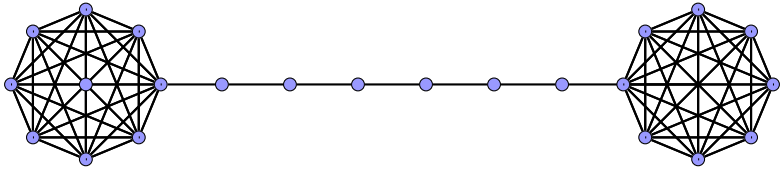

In 2013, Stani\'c \cite{Stanic2013} posed the following conjecture, 
which also was listed as an open problem in \cite{AksoyChungTaitTobin2018}, and collected in \cite[Conjecture 18]{Liu-Ning2023}.

\begin{conjecture}[\cite{Stanic2013}]
Among all connected graphs of a fixed order, the minimum adjacency spectral gap is attained by a double kite. 
\end{conjecture}

The conjecture was initially confirmed by Stani\'c for connected graphs with up to 
$10$ vertices, who also provided some contributions toward the conjecture. In \cite{JKS}, a similar conjecture is made for trees. In this paper, we prove Stani\'c's conjecture for sufficiently large $n$.


\begin{theorem}\label{thm:main}
For all sufficiently large $n$, the connected graph $G$ on $n$ vertices 
with minimum spectral gap is a double kite graph.
\end{theorem}

In fact, our proof yields more precise information about the extremal graphs: for any double kite attaining the minimum spectral gap, 
the clique size is asymptotically $n/(2\log n)$, and the corresponding minimal spectral gap at most
$\exp\{-n\log n+n\log\log n+O(n)\}$.


The proof of Theorem \ref{thm:main} is quite involved.
To clarify the overall strategy, we provide the following high-level outline.
Let $G$ be a connected $n$-vertex graph minimizing the adjacency spectral gap.

1. Establish a sharp upper bound on $\gap(G)$ using double kites.
The proof begins by analyzing the eigenvector recurrence along the path of a double kite.
Taking two cliques, each of order about $n/(2\log n)$, and joining them by a long path yields 
a graph whose spectral gap is extraordinarily small:
$\gap(G)\leq\exp\{-n\log n+n\log\log n+O(n)\}$. 

2. Split the extremal graph $G$ into two petals. 
Let $\bm{y}$ be an eigenvector corresponding to $\lambda_2(G)$, and partition the 
vertices according to whether $\bm{y}$ is positive, negative, or zero. 
In Section \ref{sec:extremal-graphs}, 
an edge-deletion argument (Lemma \ref{lem:edge-deletion}) shows that every edge across which $\bm{y}$ changes 
sign must be a cut edge. We then show that the positive and negative 
vertices each induce a connected subgraph (Lemma \ref{lem:nodal-domains-connected}). 
These two subgraphs, referred to as the two 
petals, are joined either by a cut edge or through a vertex at which $\bm{y}$ has zero entry
(Lemma \ref{lem:zero-skeleton}).

3. Determine the approximate value of $\lambda_2(G)$. We first establish 
a weighted Dirichlet sum identity for $\lambda_1(G) - \lambda_2(G)$ (Lemma \ref{lem:ground-state}). 
Within each petal, we choose a vertex at which its Perron vector is maximal and a shortest path 
from the connecting cut edge to that vertex. 
Combining with the weighted Dirichlet sum, we establish a lower bound on $\lambda_1(G) - \lambda_2(G)$ (Lemma \ref{lem:petal-compression}).
Comparing this lower bound on the spectral gap with the upper bound obtained from the double kite forces
$q:=\lceil\lambda_2(G)\rceil\sim n/(2\log n)$.

4. Reduce to a rough structure.   
We show that the extremal graph $G$ 
consists of two connected graphs $B_1,B_2$, each having $q+O(1)$ vertices and bounded diameter, 
together with a connected subgraph $C$ joining them. There is exactly one edge from each $B_i$ to $C$, while
$|V(C)| = n - 2q + O(1) = \Theta(q\log q)$ and $C$ has bounded maximum degree
(Theorem \ref{thm:terminal-extraction} and Lemma \ref{lem:terminal-parameters}).

5. Prove that $C$ is a path and each $B_i$ is a clique. By analyzing the entries of 
the resolvent matrix $(xI - A(C))^{-1}$ (Lemma \ref{lem:corridor-transfer} and Lemma \ref{lem:loaded-corridor-poles}),
we prove that if $C$ is not a path, then replacing $C$ by a path produces
a connected $n$-vertex graph with a smaller spectral gap, contradicting the minimality of $G$ (Proposition \ref{prop:corridor-is-path}).  
Once $C$ is known to be a path, we compare each $B_i$ with a complete graph of the same order. 
If $B_i$ is not complete, adding missing edges increases its spectral radius, and then gives a smaller spectral gap (Lemma \ref{lem:Bi-clique}). 

6. At this point, we have proved that the extremal graph $G$ consists of two disjoint cliques joined by a long path. 
Finally, we show that the two cliques have equal order, so $G$ is a double kite.

\section{Preliminaries}
\label{sec:preliminaries}

In this section we introduce definitions and notation that will be used throughout the
paper, and prove some preliminary lemmas.

\subsection{Definitions and notation}

Let $X\subseteq V(G)$. We write $G[X]$ for the subgraph induced by  
$X$, and $G\setminus X$ for the graph obtained by deleting $X$. 
For a vertex $v$ of $G$, $d_G(v)$ and $N_G(v)$ denote its degree and 
its neighbors, respectively. The \emph{distance} between two vertices $u$ and $v$ is
written $\mathrm{dist}_G(u,v)$. When the underlying 
graph $G$ is clear from the context, we will omit the subscript $G$ from the notation above. Let $N_X(v)$ denote the set of 
vertices in $X$ that are adjacent to $v$, i.e., $N_X(v)=N_G(v)\cap X$. Given two vertices 
$u$, $v$, we use $u\sim v$ (resp. $u\nsim v$) to indicate that vertices $u$ and $v$ are adjacent 
(resp. nonadjacent). Let $\bm{x}\in\mathbb{R}^n$ be a vector. We use 
$x_u$ or $x(u)$ to denote the entry of $\bm{x}$ corresponding to $u$.
We use $\bm{e}_u$ to denote the unit vector with $u$-th entry is one and zero otherwise, whose dimension is to be understood from the context.
For notation and graph terminology undefined here,
we refer the reader to \cite{Bondy-Murty2008}.

The Perron--Frobenius theorem implies that the adjacency matrix $A(G)$ of a graph $G$ 
has a nonnegative eigenvector corresponding to $\lambda_1(G)$, and this is 
called the \emph{Perron vector} of $G$. 
Let $\bm{x}$ be an eigenvector of $G$ corresponding to an eigenvalue $\lambda$ of $G$. For each $v\in V(G)$, 
we have $\lambda x_v = \sum_{u\in N(v)} x_u$,
and refer this as the \emph{eigenvalue\,--\,eigenvector equation} for $\lambda$.

We will make extensive use of the resolvent of a matrix. Given a square matrix $M$, define the {\em resolvent}
\[
R_M(x) = (xI - A)^{-1}.
\]
When $M$ is clear we will drop the subscript. The resolvent of an adjacency matrix captures properties of the graph, its subgraphs under vertex deletion, and its eigenvectors. If $A$ has eigenvalues $\lambda_k$ with eigenvectors $\mathbf{z}_k$ then $R_A(x)_{ij} = \sum_k \frac{\mathbf{z}_k(i) \mathbf{z}_k(j)}{x-\lambda_k}$, see \cite{Horn-Johnson2012} Chapter 2. This matrix is also closely related to the walk generating function (see \cite{Godsil} Chapter 3), since 
\[
\sum_{k\geq 0} (A^k)t^k = (I-tA)^{-1}.
\]

\subsection{Preliminary lemmas}

We begin with a weighted Dirichlet sum identity for the spectral gap of a connected graph.

\begin{lemma}\label{lem:ground-state}
Let $G$ be a connected $n$-vertex graph. 
If $\bm{x}$ is a positive unit $\lambda_1(G)$-eigenvector,
$\bm{y}$ is a unit $\lambda_2(G)$-eigenvector, and $f_v := y_v/x_v$ for $v\in V(G)$, then
\begin{equation}\label{eq:ground-state}
\lambda_1(G) - \lambda_2(G) = \sum_{uv\in E(G)} x_ux_v (f_u - f_v)^2.
\end{equation}
If $D=\operatorname{diam}(G)$, then
$\lambda_1(G) - \lambda_2(G) \geq\frac{2}{D(n-1) \lambda_1(G)^{D-1}}$.
\end{lemma}

\begin{proof}
Set, for ease of notation, $\lambda_1:=\lambda_1(G)$ and $\lambda_2:=\lambda_2(G)$.
Expanding the right-hand side of \eqref{eq:ground-state} and using
$A\bm{x}=\lambda_1\bm{x}$ and $A\bm{y}=\lambda_2\bm{y}$ gives
\begin{align}
\sum_{uv\in E(G)} x_ux_v (f_u - f_v)^2
& = \sum_{uv\in E(G)} x_ux_v (f_u^2 + f_v^2 - 2f_uf_v) \nonumber \\
& = \sum_{uv\in E(G)} (x_ux_vf_u^2 + x_ux_vf_v^2) - \bm{y}^{\top} A(G)\bm{y}. \label{eq:sum-f}
\end{align}
For the first term, collect all contributions involving a fixed vertex $v$:
\begin{align*}
\sum_{uv\in E(G)} (x_ux_vf_u^2 + x_ux_vf_v^2) 
& = \sum_{v\in V(G)} x_vf_v^2 \sum_{u\sim v} x_u = \sum_{v\in V(G)} x_vf_v^2 (\lambda_1 x_v) \\
& = \lambda_1\sum_{v\in V(G)} x_v^2 f_v^2 = \lambda_1.
\end{align*}
Together with \eqref{eq:sum-f} and $\lambda_2 = \bm{y}^{\top} A(G)\bm{y}$, we obtain \eqref{eq:ground-state}.

Next, we consider the lower bound on $\lambda_1(G) - \lambda_2(G)$. Fix vertices $i,j$, and let
$i=v_0,v_1,\ldots,v_\ell=j$ be a shortest path between $i$ and $j$. Repeated use of the
eigenvalue\,--\,eigenvector equations gives
$x_{v_k}\geq x_i/\lambda_1^k$ and $x_{v_{k}}\geq x_j/\lambda_1^{\ell-k}$ for $0\leq k\leq\ell$,
and hence $x_{v_k} x_{v_{k+1}} \geq x_ix_j/\lambda_1^{\ell-1}$.
Set $w_k:=x_{v_k}x_{v_{k+1}}>0$. The Cauchy--Schwarz inequality and \eqref{eq:ground-state} give
\begin{align*}
(f_i - f_j)^2 
& = \Bigg(\sum_{k=0}^{\ell-1} \frac{1}{\sqrt{w_k}} \sqrt{w_k} (f_{v_k} - f_{v_{k+1}})\Bigg)^2 \\
& \leq \Bigg(\sum_{k=0}^{\ell-1} \frac{1}{w_k}\Bigg) \Bigg(\sum_{k=0}^{\ell-1} w_k(f_{v_k} - f_{v_{k+1}})^2 \Bigg) \\
& \leq \frac{\ell \lambda_1^{\ell-1}}{x_ix_j} (\lambda_1 - \lambda_2).
\end{align*}
It follows that $x_i^2 x_j^2 (f_i - f_j)^2 \leq \ell (\lambda_1 - \lambda_2) \lambda_1^{\ell-1}x_ix_j
\leq D(\lambda_1 - \lambda_2)\lambda_1^{D-1}x_ix_j$. Hence, 
\[
\sum_{i<j} x_i^2x_j^2 (f_i - f_j)^2
\leq D(\lambda_1 - \lambda_2)\lambda_1^{D-1} \sum_{i<j} x_ix_j.
\]
Finally, we estimate both sides of the inequality: 
a lower bound for the left-hand side and an upper bound for the right-hand 
side. For the left-hand side, since $\langle\bm{x}, \bm{y}\rangle = 0$, we deduce
\[
\sum_{i<j} x_i^2x_j^2(f_i-f_j)^2 = \Big(\sum_i x_i^2\Big)
\Big(\sum_i x_i^2f_i^2\Big) - \Big(\sum_i x_i^2f_i\Big)^2 = 1,
\]
while for the right-hand side, the Cauchy--Schwarz inequality gives 
$\sum_{i<j} x_ix_j = (\|\bm{x}\|_1^2 - 1)/2 \leq (n-1)/2$.
Combining these bounds proves the desired inequality.
\end{proof}

The following lemma is known as the Poincar\'e 
Separation Theorem \cite[Corollary 4.3.37]{Horn-Johnson2012}.

\begin{lemma}[{\cite[Corollary 4.3.37]{Horn-Johnson2012}}]\label{lem:interlace-inequality}
Let $A$ be a symmetric $n\times n$ matrix with eigenvalues $\lambda_1\geq\cdots\geq\lambda_n$. 
Suppose that $1 \leq m \leq n$, and let $\bm{u}_1,\ldots,\bm{u}_m\in\mathbb{R}^n$ 
be orthonormal. Define $B = [\bm{u}_i^{\top} A\bm{u}_j]_{i,j=1}^m$ 
and let $B$ have eigenvalues $\mu_1\geq\cdots\geq\mu_m$. Then 
$\lambda_{i+n-m} \leq\mu_i\leq\lambda_i$, $i=1,2,\ldots,m$.
\end{lemma}

We will also employ the standard Hoffman--Smith subdivision lemma \cite{HoffmanSmith}. 
The exceptional affine Dynkin graph in the general statement has spectral radius $2$, and is therefore excluded by
the hypothesis below.

\begin{lemma}[{\cite[Theorem 8.1.12]{Cvetkovic-Rowlinson-Simic}}]\label{lem:HS-subdivision}
Let $H$ be a connected graph with $\lambda_1(H) > 2$, and let $e$ lie on an internal
path, that is, on a path whose endpoints have degree at least $3$ and whose
internal vertices have degree $2$. If $H'$ is obtained by subdividing
$e$ once, then $\lambda_1(H') < \lambda_1(H)$.
\end{lemma}

\section{Eigenvalues of double-kite graphs}
\label{sec:double-kite}

In this section, we derive several eigenvalue properties of double-kite graphs. The results in this section essentially come from the arguments in \cite{Cioaba-Gregory} and \cite{Stanic2013}.
For $j\geq 0$, define the polynomial $p_j(x)$ recursively by
\begin{equation}\label{eq:path-polynomials}
p_0(x) = 1, \qquad p_1(x) = x, \qquad
p_j(x) = xp_{j-1}(x) - p_{j-2}(x) \quad (j\geq 2).
\end{equation}
For the path $P_j$ on $j$ vertices, expanding $\det(xI-A(P_j))$ along an endvertex gives
$p_j(x)=\phi(P_j,x)$, where $\phi(P_j, x)$ is the characteristic polynomial of $P_j$\footnote{Moreover, $p_j(x)=U_j(x/2)$, where $U_j$ is the Chebyshev polynomial of
the second kind.}. For $x > 2$, put $\sigma(x):=(x+\sqrt{x^2-4})/2$.

\subsection{A recursive polynomial}

The following lemma gives a closed form of $p_j(x)$.

\begin{lemma}
For every integer $j\geq0$ and every real $x>2$,
\begin{equation}\label{eq:path-hyperbolic}
p_j(x) = \frac{\sigma(x)^{j+1}-\sigma(x)^{-j-1}}{\sigma(x)-\sigma(x)^{-1}}.
\end{equation}
\end{lemma}

\begin{proof}
By \eqref{eq:path-polynomials}, the two roots of the characteristic equation
$\xi^2-x\xi+1 = 0$ are $\sigma(x)$ and $\sigma(x)^{-1}$. Hence, $p_j(x)$ can be written as 
$p_j(x) = a\sigma(x)^j + b\sigma(x)^{-j}$ for some $a,b$. Using $p_0(x) = 1$ 
and $p_1(x) = x$, we obtain $a + b = 1$ and $a\sigma(x) + b\sigma(x)^{-1} = x$.
Solving it and note that $\sigma(x)+\sigma(x)^{-1}=x$, we have 
\[
a = \frac{\sigma(x)}{\sigma(x) - \sigma(x)^{-1}}, \qquad
b = -\frac{\sigma(x)^{-1}}{\sigma(x) - \sigma(x)^{-1}}.
\]
Substituting them into $p_j(x) = a\sigma(x)^j + b\sigma(x)^{-j}$ gives \eqref{eq:path-hyperbolic}.
\end{proof}

We use the following notation. 
Let $F$ be a graph. For $x > \lambda_1(F)$, let
$R_F(x) := (xI - A(F))^{-1}$.

\begin{lemma}\label{lem:R(F-x)}
Let $x>2$ and $v_1, v_s$ be the two leaf of the $s$-vertex path. Then 
\[
\bm{e}_{v_1}^{\top} R_{P_s} (x) \bm{e}_{v_1} = \frac{p_{s-1}(x)}{p_s(x)}, \qquad
\bm{e}_{v_1}^{\top} R_{P_s} (x) \bm{e}_{v_s} = \frac{1}{p_s(x)}.
\]
\end{lemma}

\begin{proof}
Consider the spectral decomposition $A(P_s)
= U\mathrm{diag}(\lambda_1,\ldots,\lambda_s)U^\top$, where
$U$ is an orthogonal matrix, and $\lambda_i$ are the eigenvalues of $P_s$, $i=1,2,\ldots,s$. Hence
\[
R_{P_s} (x) = U\mathrm{diag}
\Big(\frac1{x - \lambda_1},\ldots, \frac{1}{x - \lambda_s}\Big)U^\top.
\]
For an invertible matrix $M$, Cramer's rule gives
$(M^{-1})_{ij} = \frac{\mathrm{cof}_{ji}(M)}{\det M}$,
where $\mathrm{cof}_{ji}(M)$ denotes the $(j,i)$-cofactor. Take
$M = xI - A(P_s)$. For the \((1,1)\)-entry of $(xI - A(P_s))^{-1}$,
\[
\bm{e}_{v_1}^{\top} R_{P_s} (x) \bm{e}_{v_1} = (R_{P_s} (x))_{11}
= \frac{\mathrm{cof}_{11}(M)}{\det M}.
\]
Deleting the first row and first column of $M$ leaves
$xI - A(P_{s-1})$. Therefore
$\mathrm{cof}_{11}(M) = \det(xI-A(P_{s-1})) = p_{s-1}(x)$.
Since $\det M = p_s(x)$, we obtain $\bm{e}_{v_1}^{\top} R_{P_s} (x) \bm{e}_{v_1} = 
p_{s-1}(x)/p_s(x)$.

Again by Cramer's rule, we have
\[
\bm{e}_{v_1}^{\top} R_{P_s} (x) \bm{e}_{v_s} = (R_{P_s} (x))_{1s}
= \frac{\mathrm{cof}_{s1}(M)}{\det M}.
\]
The cofactor is $\mathrm{cof}_{s1}(M) =
(-1)^{s+1}\det M_{s1} = 1$, where $M_{s1}$ is obtained from $M$ by deleting row $s$ and column $1$, and $\det M_{s1}=(-1)^{s-1}$.
It follows that
$\bm{e}_{v_1}^{\top} R_{P_s} (x) \bm{e}_{v_s}  = 1/p_s(x)$.
\end{proof}

\subsection{The gap of double-kite graphs}

In the following lemma, we determine the equation that the first two largest eigenvalues of 
$DK(r,s)$ must satisfy. Let
\begin{equation}\label{eq:clique-response}
q_r(x) := x-\frac{r-1}{x-r+2}.
\end{equation}

\begin{lemma}\label{lem:DK-equations}
Let $r\geq 5$ and $s\geq 2$. The largest and second-largest eigenvalues of $DK(r,s)$
are the unique roots in $(r-1, \infty)$ of the following equations, respectively:
\begin{equation}\label{eq:DK-plus-minus}
q_r(x) = \frac{p_{s-1}(x)+1}{p_s(x)}, \qquad
q_r(x) = \frac{p_{s-1}(x)-1}{p_s(x)}.
\end{equation}
\end{lemma}

\begin{proof}
Let $u_0,u_1,\ldots,u_s,u_{s+1}$ be the vertices along the path of $DK(r,s)$ 
between the two attachment vertices $u_0$ and $u_{s+1}$. Now consider an eigenvalue 
$\lambda > r-1$ of $DK(r,s)$ with a corresponding eigenvector $\bm{z}$, 
and let $z_i$ be the coordinate of $\bm{z}$ at $u_i$. From the eigenvalue\,--\,eigenvector equation for $\lambda$, and since $\lambda_1$ and $\lambda_2$ are simple,
any two non-attachment vertices in the same clique of $DK(r,s)$ have the same coordinate; 
we denote this common coordinate by $b_0$.
Moreover, $\lambda\, b_0 = z_0 + (r-2)b_0$, which rearranges to $b_0 = z_0/(\lambda - r + 2)$.
The eigenvalue\,--\,eigenvector equation at $u_0$ is $\lambda\, z_0 = (r-1)b_0 + z_1$.
Combining this yields
\begin{equation}\label{eq:DK-left-boundary}
z_1 = \Big(\lambda - \frac{r-1}{\lambda - r + 2}\Big) z_0 = q_r(\lambda) z_0.
\end{equation}
Moreover, one necessarily has $z_0\neq 0$. Likewise, we have $z_s = q_r(\lambda) z_{s+1}$.

For $1\leq i\leq s$, the eigenvalue\,--\,eigenvector equation at $u_i$ is
$z_{i+1} = \lambda z_i - z_{i-1}$. We claim more generally that by induction
$z_{k+1} = p_k(\lambda) z_1 - p_{k-1}(\lambda) z_0$ for $1\leq k\leq s$.
For $k=1$, this is clear. If the claim holds for $k-1$ and $k$, then
$z_{k+2} = \lambda z_{k+1} - z_k = p_{k+1}(\lambda) z_1 - p_k(\lambda) z_0$
by the recurrence \eqref{eq:path-polynomials}. This proves the claim. In particular,
\begin{equation}\label{eq:path-transfer-DK}
z_{s+1} = p_s(\lambda) z_1 - p_{s-1}(\lambda) z_0.
\end{equation}

Let $J$ be the permutation matrix representing the reflection that exchanges 
the two cliques and reverses the joining path. Then $J^{\top} = J$ and $J^2 = I$.
Since the reflection is a graph automorphism, $J^{\top} AJ = A$, which is equivalent to 
$AJ = JA$. As $J^2 = I$, its eigenvalues belong to $\{1, -1\}$, and 
$\mathbb{R}^n = \ker (J - I)\oplus\ker (J + I)$. Since $AJ = JA$, 
both subspaces $\ker (J - I)$ and $\ker (J + I)$ are invariant under $A$, 
which allows $A$ to be diagonalized separately on them.  We may therefore
choose the eigenvector $\bm z$ for $\lambda$ in one of these two subspaces.
Then $J\bm{z} = \varepsilon \bm{z}$, where
$\varepsilon\in\{1,-1\}$.
Since reflection reverses this path, $J \bm{e}_{u_0} = \bm{e}_{u_{s+1}}$.
Hence, 
\[
z_{s+1} = \bm{e}_{u_{s+1}}^{\top} \bm{z} = (J \bm{e}_{u_0})^{\top} \bm{z} 
= \bm{e}_{u_0}^{\top} J^{\top} \bm{z} = \bm{e}_{u_0}^{\top} J \bm{z} 
= \bm{e}_{u_0}^{\top} (\varepsilon \bm{z}) = \varepsilon z_0.
\]
Then $z_{s+1} = \varepsilon z_0$. Since $z_0\neq 0$, we may therefore divide \eqref{eq:path-transfer-DK} by $z_0$. Together with \eqref{eq:DK-left-boundary}, we obtain
$\varepsilon = p_s(\lambda) q_r(\lambda) - p_{s-1}(\lambda)$, or equivalently
$q_r(\lambda) = (p_{s-1}(\lambda) + \varepsilon)/p_s(\lambda)$.
The choices $\varepsilon=1$ and $\varepsilon=-1$ imply that $\lambda$ 
is a roots of \eqref{eq:DK-plus-minus}.
\smallskip

Set $R_s(x) := (xI - A(P_s))^{-1}$, $x>2$ for short. By Lemma \ref{lem:R(F-x)}, we see
\[
(R_s(x))_{11}=\frac{p_{s-1}(x)}{p_s(x)}, \qquad
(R_s(x))_{1s}=\frac1{p_s(x)}.
\]
Reflection symmetry gives
$(R_s(x))_{ss} = (R_s(x))_{11}$ and $(R_s(x))_{s1} = (R_s(x))_{1s}$. Therefore, with $h_\pm(x) := \frac{p_{s-1}(x)\pm 1}{p_s(x)}$, we have
\begin{equation}\label{eq:h-resolvent-form}
h_\pm(x) = \frac{1}{2} (\bm{e}_1\pm \bm{e}_s)^{\top} R_s(x) (\bm{e}_1\pm \bm{e}_s) > 0.
\end{equation}
Differentiating the inverse matrix gives $R_s'(x) = -R_s(x)^2$. Consequently,
\begin{equation}\label{eq:h-strictly-decreasing}
h_\pm'(x) = -\frac{1}{2} \|R_s(x) (\bm{e}_1\pm \bm{e}_s)\|^2 < 0.
\end{equation}
On the other hand,
\begin{equation}\label{eq:q-strictly-increasing}
q_r'(x) = 1 + \frac{r-1}{(x-r+2)^2} > 0.
\end{equation}
At $x=r-1$, one has $q_r(r-1)=0$, whereas
$h_\pm(r-1)>0$ by \eqref{eq:h-resolvent-form}. As
$x\to\infty$, one has $q_r(x)\to\infty$ and $h_\pm(x)\to0$.
Thus each of the equations $q_r(x) = h_+(x)$ and $q_r(x) = h_-(x)$ has a root in
$(r-1,\infty)$. Moreover, \eqref{eq:h-strictly-decreasing} and
\eqref{eq:q-strictly-increasing} show that each root is unique.

Finally, we show that these are the two largest eigenvalues.
Delete $u_0$ and $u_{s+1}$ from $DK(r,s)$. The remaining induced
subgraph is $K_{r-1}\sqcup P_s\sqcup K_{r-1}$, whose
spectral radius is $\max\{\lambda_1(K_{r-1}), \lambda_1(P_s)\} = r-2$,
since $r\geq 5$ and $\lambda_1(P_s) < 2$. By Cauchy interlacing, deleting two vertices gives $\lambda_3 (DK(r,s))\leq r-2$. The two roots obtained above both exceed 
$r-1$, and hence also exceed $\lambda_3(DK(r,s))$. Therefore they are the two largest adjacency eigenvalues.
\end{proof}

\begin{lemma}\label{lem:DK-gap-upper}
Let $r\geq 5$ and $s\geq 2$. Then $\gap(DK(r,s))
\leq \frac{4}{r\,p_s(r-1)}\leq \frac{4}{r}\,\sigma(r-1)^{-s}$.
\end{lemma}

\begin{proof}
We first record two elementary estimates for the path polynomials. For
$x\geq 3$, put $\rho_j(x):=\frac{p_{j-1}(x)}{p_j(x)}$.
Since $\rho_1=1/x\in(0,1)$, the recurrence gives inductively
\[
\rho_j(x) = \frac1{x-\rho_{j-1}(x)} \quad\text{and hence}\quad
0 < \rho_j(x)\leq\frac{1}{x-1} < 1.
\]
It follows in particular that $p_j(x)>p_{j-1}(x)>0$ for $j\geq 2$. Therefore, for $s\geq 2$ we have
\begin{equation}\label{eq:path-ratio-elementary-bounds}
\frac{p_{s-1}(x)}{p_s(x)}\leq\frac{1}{x-1}, \qquad
\frac{1}{p_s(x)}\leq\frac{1}{x}.
\end{equation}
At $x=r-1$, these inequalities give $h_+(r-1) \leq 2/(r-2)$.

Write either of the two roots of \eqref{eq:DK-plus-minus} as $x=r-1+a$, where $a>0$.  Since both
$h_+$ and $h_-$ are decreasing and $h_-\leq h_+$, Lemma \ref{lem:DK-equations} implies
\begin{equation}\label{eq:DK-root-shift-initial}
 \frac{a(r+a)}{1+a}
 =q_r(r-1+a)
 \leq h_+(r-1)
 \leq\frac2{r-2}.
\end{equation}
If $a\geq1$, then the left-hand side is at least $r/2$, whereas the
right-hand side is at most $2/3$; this is impossible for $r\geq5$.
Thus $a<1$.  Using $r+a\geq r$ and $1+a<2$ in
\eqref{eq:DK-root-shift-initial}, we obtain the sharper estimate
$0 < a\leq\frac4{r(r-2)} < \frac13$.
Consequently, every $x$ between $\lambda_1:=\lambda_1(DK(r,s))$ and $\lambda_2:=\lambda_2(DK(r,s))$ satisfies
$x-r+2<4/3$. By \eqref{eq:q-strictly-increasing}, $q_r'(x)\geq 1 + 9(r-1)/16 \geq r/2$.
The mean-value theorem and the fact that $h_+$ is decreasing now give
\begin{align*}
\frac{r}{2} (\lambda_1 - \lambda_2)
& \leq q_r(\lambda_1) - q_r(\lambda_2) = h_+(\lambda_1) - h_-(\lambda_2) \\
& \leq h_+(\lambda_2) - h_-(\lambda_2) = \frac{2}{p_s(\lambda_2)}.
\end{align*}
By \eqref{eq:path-hyperbolic}, $p_s(x)$ is positive and strictly
increasing on $(2,\infty)$. Since $\lambda_2 > r - 1$, it follows that
$p_s(\lambda_2) \geq p_s(r-1)$. We have therefore proved
$\lambda_1 - \lambda_2\leq\frac4{r\,p_s(r-1)}$.

Finally, \eqref{eq:path-hyperbolic} can be rewritten as
\begin{equation}\label{eq:path-hyperbolic-factorized}
 p_s(x)
 =\sigma(x)^s
   \frac{1-\sigma(x)^{-2s-2}}{1-\sigma(x)^{-2}}.
\end{equation}
Since $\sigma(x)>1$ and $s\geq0$, the quotient on the right is at
least $1$. Thus $p_s(x)\geq\sigma(x)^s$, which in fact gives the desired estimation.
\end{proof}

\section{The properties of extremal graphs}
\label{sec:extremal-graphs}

For the remainder of the paper, we always assume $G$ is a graph attaining the minimum spectral gap
among all $n$-vertex connected graphs, and let $\gamma_n := \gap(G)$. 
We tacitly assume that $n$ is sufficiently large. Let $\bm{x}$ be the positive unit eigenvector 
corresponding to $\lambda_1(G)$, and $\bm{y}$ be a unit eigenvector of $G$ corresponding to
$\lambda_2(G)$. For short, we set $\lambda_1:=\lambda_1(G)$, $\lambda_2:=\lambda_2(G)$.

\begin{lemma}\label{lem:magnitudes-order}
$\gamma_n\leq\exp\{-n\log n+n\log\log n+O(n)\}$.
In particular, $\gamma_n = o(n^{-K})$ for any fixed $K > 0$.
\end{lemma}

\begin{proof}
Take $r=\lfloor n/(2\log n)\rfloor$ and $s=n-2r$.  
Since $\sigma(x)=x+O(x^{-1})$ as $x\to\infty$, we deduce
\[
\log\sigma(r-1) = \log r + O(r^{-1}) = \log n - \log\log n + O(1).
\]
Moreover, $s = n - O(n/\log n)$. Taking logarithms in Lemma \ref{lem:DK-gap-upper} therefore gives
\[
\log\gamma_n \leq \log(4/r) - s\log\sigma(r-1)
= - n\log n + n\log\log n + O(n),
\]
which proves the desired inequality.
\end{proof}

\subsection{The second largest eigenvalue and its eigenvector}
\label{sec:nodal-skeleton}

\begin{proposition}\label{prop:t-greater-two-nontree}
The second largest eigenvalue satisfies $\lambda_2(G) > 2$, and $G$ is not a tree.
\end{proposition}

\begin{proof}
If $\lambda_2\leq 2$, then $\lambda_1 = \lambda_2 + \gamma_n < 3$ by
Lemma~\ref{lem:magnitudes-order}, while Lemma~\ref{lem:ground-state} gives
$\lambda_1 - \lambda_2\geq 2(n-1)^{-2} 3^{-(n-2)}$,
contradicting Lemma \ref{lem:magnitudes-order}. If $G$ is a tree, then $\lambda_1^2\leq\operatorname{tr} A(G)^2 = 2(n-1)$.
Thus Lemma \ref{lem:ground-state} gives
\[
\lambda_1 - \lambda_2\geq\frac{2}{(n-1)^2[2(n-1)]^{(n-2)/2}}
= \exp\{-\tfrac12 n\log n+O(n)\},
\]
again contradicting Lemma \ref{lem:magnitudes-order}.
\end{proof}

\begin{lemma}\label{lem:edge-deletion}
Let $uv\in E(G)$. If $uv$ is not a cut edge, then $y_uy_v > 0$.
\end{lemma}

\begin{proof}
Suppose by contradiction that $y_uy_v \leq 0$. 
Let $A$ be the adjacency matrix of $G$. Write $\bm{y}_+$ and $\bm{y}_-$ for the 
positive and negative parts of $\bm{y}$, respectively; thus both vectors are nonnegative and   
$\bm{y} := \bm{y}_+ - \bm{y}_-$, $\langle \bm{y}_+, \bm{y}_-\rangle = 0$. Put $c := \bm{y}_+^{\top} A \bm{y}_- \geq 0$.
Since $A\bm{y} = \lambda_2 \bm{y}$, we see $(A - \lambda_2 I)(\bm{y}_+ - \bm{y}_-) = 0$. 
Hence $(A - \lambda_2I) \bm{y}_+ = (A - \lambda_2I) \bm{y}_-$.
Multiplying on the left by $\bm{y}_+^{\top}$ and $\bm{y}_-^{\top}$, respectively, we obtain
\[
\bm{y}_+^{\top} (A - \lambda_2I) \bm{y}_+ = \bm{y}_+^{\top} (A - \lambda_2I)\bm{y}_-, \quad 
\bm{y}_-^{\top} (A - \lambda_2I) \bm{y}_+ = \bm{y}_-^{\top} (A - \lambda_2I) \bm{y}_-.
\]
Together with $\langle \bm{y}_+, \bm{y}_-\rangle = 0$,
\begin{equation}\label{eq:pq-forms}
\bm{y}_+^{\top} (A - \lambda_2I) \bm{y}_+ = \bm{y}_-^{\top} (A-\lambda_2I)\bm{y}_- 
= \bm{y}_+^{\top} (A - \lambda_2I)\bm{y}_- = c.
\end{equation}

Suppose, after interchanging $u$ and $v$ if necessary, that
$y_u\geq 0 \geq y_v$, and set $d := (y_+)_u(y_-)_v\geq 0$. Let $M = A(G-uv) - \lambda_2I$. Then 
\[
\begin{bmatrix}
\bm{y}_+^{\top} M\bm{y}_+ & \bm{y}_+^{\top} M \bm{y}_- \\
\bm{y}_-^{\top} M\bm{y}_+ & \bm{y}_-^{\top} M\bm{y}_-
\end{bmatrix} =
\begin{bmatrix}
c & c-d \\
c-d & c
\end{bmatrix}.
\]
This matrix is positive semidefinite, since $c\geq 0$ and its determinant $d(2c-d)\geq 0$;
here $c\geq d$, since $c = \bm{y}_+^{\top} A\bm{y}_-$. Applying Lemma \ref{lem:interlace-inequality} then yields
$\lambda_2(G-uv) \geq \lambda_2$.
Since $uv$ is not a cut edge, the graph $G-uv$ is connected. Consequently, $\lambda_1(G-uv) < \lambda_1$.
It follows that $\lambda_1(G-uv) - \lambda_2(G - uv) < \lambda_1 - \lambda_2$, 
which contradicts the minimality of $G$. This proves $y_uy_v > 0$.
\end{proof}

By Lemma~\ref{lem:edge-deletion}, every edge joining
opposite signs is a cut edge; and every edge incident with a vertex of $Z$ is a cut edge.

\begin{lemma}\label{lem:lambda2-simple}
The eigenvalue $\lambda_2(G)$ is simple. 
\end{lemma}

\begin{proof}
By Proposition~\ref{prop:t-greater-two-nontree}, $G$ has a non-cut edge $uv$.
If $\lambda_2$ had multiplicity at least two, its eigenspace would contain a
nonzero vector vanishing at $u$, contradicting Lemma~\ref{lem:edge-deletion}.
Thus $\lambda_2$ is simple.
\end{proof}

In the next subsection, we partition $V(G)$ into three subsets according to the sign of the entries of 
$\bm{y}$, and investigate relevant properties. Let
\[
P := \{v\in V(G): y_v > 0\}, \quad N := \{v\in V(G): y_v < 0\}, \quad Z := \{v\in V(G): y_v = 0\}.
\]

\subsection{Local structure}

We will need the following result, which appears as a special case of a more general result in \cite{Urschel2018}; for convenience, we reproduce a proof here.

\begin{lemma}[\cite{Urschel2018}]\label{lem:nodal-domains-connected}
The induced subgraphs $G[P]$ and $G[N]$ are both connected.
\end{lemma}

\begin{proof}
Let $S_1,\ldots,S_m$ be the connected components of $G[P]$ and $G[N]$,
and let $\bm{y}_i$ denote the restriction of $\bm{y}$ to $S_i$, extended by zero
outside $S_i$. Then $\bm{y} = \sum_{i=1}^m \bm{y}_i$.
For arbitrary real numbers $c_1,c_2,\ldots,c_m$, put
$\bm{z} := \sum_{i=1}^m c_i \bm{y}_i$.

Since the supports of $\bm{y}_i$ and $\bm{y}_j$ are disjoint when $i\neq j$, we have $\langle \bm{y}^{(i)}, \bm{y}^{(j)}\rangle = 0$, and therefore
$\bm{y}_i^{\top} (A - \lambda_2 I) \bm{y}_j = \bm{y}_i^{\top} A \bm{y}_j$ ($i\neq j$).
Since $(A - \lambda_2 I) \bm{y} = 0$, it follows that
$0 = \bm{y}_i^{\top} (A - \lambda_2I) \bm{y} = \bm{y}_i^{\top} (A - \lambda_2I) \sum_{j=1}^m \bm{y}_j$. 
Hence, for each $i$, 
\begin{equation}\label{eq:y(A-tI)y}
\bm{y}_i^{\top} (A - \lambda_2I) \bm{y}_i = -\sum_{j\neq i} \bm{y}_i^{\top} (A - \lambda_2I) \bm{y}_j 
= -\sum_{j\neq i} \bm{y}_i^{\top} A \bm{y}_j.
\end{equation}
Now expand the quadratic form of $\bm{z}$, we deduce that
\begin{align*}
\bm{z}^{\top} (A - \lambda_2I) \bm{z}
& = \bigg(\sum_i c_i \bm{y}_i\bigg)^{\top} (A - \lambda_2I) \bigg(\sum_j c_j \bm{y}_j\bigg) \\
& = \sum_i c_i^2 \bm{y}_i^{\top} (A - \lambda_2I) \bm{y}_i
+ 2\sum_{i<j} c_ic_j \bm{y}_i^{\top} (A - \lambda_2I) \bm{y}_j.
\end{align*}
Using $\bm{y}_i^{\top} (A - \lambda_2I) \bm{y}_j = \bm{y}_i^{\top} A \bm{y}_j$ for $i\neq j$ and applying \eqref{eq:y(A-tI)y}, we obtain
\begin{align}
\bm{z}^{\top} (A - \lambda_2I) \bm{z} 
& = -\sum_i c_i^2 \sum_{j\neq i} \bm{y}_i^{\top} A \bm{y}_j + 
2\sum_{i<j} c_ic_j \bm{y}_i^{\top} A \bm{y}_j \nonumber \\
& = - \sum_{i<j} (c_i^2 + c_j^2) \bm{y}_i^{\top} A \bm{y}_j 
+ 2\sum_{i<j} c_ic_j \bm{y}_i^{\top} A \bm{y}_j \nonumber \\
& = -\sum_{i<j} (c_i-c_j)^2\,\bm{y}_i^{\top} A \bm{y}_j\geq 0. \label{eq:unified-nodal-form}
\end{align}
The last inequality holds since distinct components of the same sign
have no edges between them, whereas coordinates in components of opposite
signs have negative products.

If $m\geq3$, we may choose nonzero $c_1,\ldots,c_m$ so that
$\bm{z}\perp \bm{x}$ and $\bm{z}\perp \bm{y}$. On the subspace $\bm{x}^\perp$, the quadratic form
associated with $A - \lambda_2I$ is nonpositive; indeed, Courant--Fischer theorem gives
for $\bm{z}\perp \bm{x}$, $\bm{z}^T A \bm{z} \leq \lambda_2 \bm{z}^T\bm{z}$, which is precisely $\bm{z}^T(A - \lambda_2I) \bm{z}\leq 0$.
Thus equality holds in \eqref{eq:unified-nodal-form}, and
$\bm{z}$ lies in the $\lambda_2$-eigenspace. Since $\lambda_2$ is simple by Lemma \ref{lem:lambda2-simple}, together with
$\bm{z}\perp \bm{y}$, gives $\bm{z} = 0$, a contradiction. Hence $m\leq 2$. Together with the fact that $m\geq 2$, this yields $m=2$.
\end{proof}

The induced connected subgraphs $G[P]$ and $G[N]$ are both connected. 
For brevity, we call them the two \emph{petals} of $G$.

\begin{lemma}\label{lem:zero-skeleton}
Exactly one of the following holds.
\begin{enumerate}
\item[$(1)$] $Z = \emptyset$, and $e(P, N) = 1$.
\item[$(2)$] $Z\neq\emptyset$, and $e(P, N) = 0$. There is a unique
vertex $w\in Z$ having neighbors outside $Z$. It has precisely one neighbor
in $P$ and one in $N$; $G[Z]$ is a tree rooted at $w$. Moreover, $\lambda_1(G[P]) = \lambda_1(G[N]) = \lambda_2$.
\end{enumerate}
\end{lemma}

\begin{proof}
(1) If $Z=\emptyset$, since $G$ is connected, we have $e(P, N)\geq 1$. 
We show that $e(P, N)\geq 2$ is impossible. If there are two edges
in $E(P,N)$, by Lemma \ref{lem:edge-deletion}, the two edges must be cut edges.
However, $G[P]$ and $G[N]$ are both connected by Lemma \ref{lem:nodal-domains-connected},
the two such edges must lie on a cycle, a contradiction. This proves (1).

(2) Now suppose that $Z\neq\emptyset$. Each vertex in $Z$ has at most one neighbor
in $P$ and at most one in $N$. Indeed, if a vertex $z\in Z$ had two 
neighbors in $P$ (respectively, in $N$), then these two edges, 
together with a path between the two neighbors in $G[P]$ (respectively, in $G[N]$), 
would form a cycle, contradicting Lemma~\ref{lem:edge-deletion}. 
Since $G$ is connected, there is at least one vertex $v\in Z$ that has a neighbor in $P\cup N$.
The eigenvalue\,--\,eigenvector equation at  
$v$ gives $0 = \sum_{u\sim v} y_u$. Hence, it has exactly one neighbor in $P$ and exactly 
one neighbor $N$. 

Next, we show that there cannot be two such zero vertices; indeed, paths in $P$ and $N$, together
with their four incident edges, would form a cycle, a contradiction to Lemma \ref{lem:edge-deletion}. 
Thus there is a unique such vertex, denoted $w$. The connectedness of $G$ forces all of $Z$ to lie in the component containing $w$. 
Moreover, every edge of $G[Z]$ is a cut edge by Lemma \ref{lem:edge-deletion}, and hence $G[Z]$ is
a tree. 

Finally, the restriction $\bm{y}|_P$ is positive and satisfies
$A(G[P]) \bm{y}|_P = \lambda_2\,\bm{y}|_P$, since every external neighbor has zero coordinate. 
By the Perron--Frobenius theorem for connected graphs, it follows that $\lambda_1(G[P]) = \lambda_2$. The proof for $N$, using $-\bm{y}|_N$, is identical.
\end{proof}

In light of Lemma~\ref{lem:HS-subdivision}, we can further derive additional property of $Z$.

\begin{lemma}\label{lem:no-zero-branches}
If $Z\neq\emptyset$, then $|Z| = 1$.
\end{lemma}

\begin{proof}
Suppose, for a contradiction, that $|Z|\geq 2$, and let $w\in Z$ be the vertex with a neighbor in $P$ and $N$. By Lemma \ref{lem:zero-skeleton}, 
$G[Z]$ is a tree. Since $|Z|\geq 2$, $G[Z]$ has a leaf $w'\neq w$. Set $H = G - w'$. 
The graph $H$ is connected and is a proper subgraph of $G$, so
$\lambda_1(H) < \lambda_1$.
Let $a\in P$ be the unique neighbor of $w$, and construct $H'$ from
$H$ by subdividing the edge $aw$, using $w'$ as the subdivision vertex.
The resulting graph $H'$ is again connected. We claim that $\lambda_1(H') < \lambda_1(H)$.
Indeed, Lemma \ref{lem:zero-skeleton} gives $\lambda_1(G[P]) = \lambda_1(G[N]) = \lambda_2 > 2$, 
so neither $G[P]$ nor $G[N]$ is a path or a cycle. 
Starting at $w$ and proceeding through $a$, follow a maximal path in $H$ 
whose internal vertices have degree two. This path cannot end at a 
degree-one vertex, since that would force $G[P]$ to be a path; 
it cannot close before meeting a vertex of degree at least three, for then $G[P]$ would
be a cycle. Thus it reaches a vertex of degree at least three. Applying the same
argument on $G[N]$ gives another such vertex. If
$w$ has degree at least three in $H$, it is one endpoint of
the required internal path; otherwise the two paths join through
$w$. Hence $aw$ lies on an internal path of $H$.
Furthermore, $H$ contains $G[P]$ as a proper subgraph, so
$\lambda_1(H) > \lambda_2 > 2$. It follows from Lemma \ref{lem:HS-subdivision} that
$\lambda_1(H') < \lambda_1(H) < \lambda_1$.

Deleting $w$ from $H'$ leaves one component $G[N]$, whose spectral radius is $\lambda_2$, and another component obtained from $G[P]$ by attaching the leaf $w'$
to $a$, whose spectral radius is therefore strictly greater than $\lambda_2$. Hence $\lambda_2(H'-w)\geq \lambda_2$.
By Cauchy interlacing, $\lambda_2(H')\geq \lambda_2$. Together with
$\lambda_1(H') < \lambda_1(H) < \lambda_1$, this implies that
$\gap(H') < \lambda_1 - \lambda_2$,
contradicting the minimality of $G$. 
\end{proof}

Henceforth, we write $H_1:= G[P]$ and $H_2:=G[N]$.  

\subsection{The approximate value of $\lambda_2(G)$}
\label{sec:defect}

In this subsection, we determine the approximate value of $\lambda_2(G)$. Roughly speaking, we will lower bound $\gap(G)$ in terms of eigenvector entries, then lower bound these entries, then show that having a spectral gap as small as the double kite is quite restrictive for the possible value of $\lambda_2$. To this end, we first fix some notation used in this subsection. 
When $Z = \emptyset$, we write the unique edge joining $H_1$ and $H_2$ is $e=u_1u_2$, where $u_i\in V(H_i)$.  
When $Z \neq\emptyset$, we denote the unique vertex in $Z$ is $w$ with neighbors $u_1$ and $u_2$ in $H_1$ and $H_2$. Moreover, we always 
write $h_i:=|V(H_i)|$, $\alpha_i:=\lambda_1(H_i)$.

\begin{lemma}\label{lem:petal-compression}
Let $\bm{z}_i$ be the unit Perron vector of $H_i$. If $Z = \emptyset$, then $\lambda_1(H_i) > \lambda_2$, $i=1,2$, and
$\lambda_1 - \lambda_2 > z_1(u_1) z_2(u_2)$.
If $|Z| = 1$, then $\lambda_1(H_1) = \lambda_1(H_2) = \lambda_2$ and
$\lambda_1 - \lambda_2\geq (z_1(u_1)^2 + z_2(u_2)^2)/\lambda_1$.
\end{lemma}

\begin{proof}
First assume $Z = \emptyset$. 
If $a\in P\setminus\{u_1\}$, then $a$ has no neighbor outside $P$. 
Hence, the eigenvalue\,--\,eigenvector equation at $a$ gives
\[
\big(A(G[P]) \bm{y}|_P\big)_a = \sum_{b\sim a,\, b\in P} y_b = \sum_{b\sim a} y_b 
= \lambda_2\,y_a = \lambda_2(\bm{y}|_P)_a. 
\]
At the vertex $u_1$, we have $\lambda_2\, y_u = \sum_{b\sim u,\, b\in P} y_b + y_{u_2}$. Hence, 
\[
\big(A(G[P]) \bm{y}|_P\big)_{u_1} = \sum_{b\sim u,\, b\in P} y_b 
= \lambda_2\,y_{u_1} - y_{v_1} > \lambda_2\,y_{u_1} = \lambda_2(\bm{y}|_P)_{u_1}.
\]
Since $G[P]$ is connected and $\bm{y}|_P>0$, we see $\lambda_1(G[P])>\lambda_2$.  
The argument for $G[N]$, applied to $-\bm{y}|_N$, is identical. Thus $\lambda_1(G[N]) > \lambda_2$. 

Extend $\bm{z}_1$ to all vertices of $G$ by assigning value $0$ outside $G[P]$, and
$\bm{z}_2$ to all vertices of $G$ by assigning value $0$ outside $G[N]$.
Then $\langle \bm{z}_1, \bm{z}_2\rangle = 0$. Since $\bm{z}_1$ vanishes outside $G[P]$,
$\bm{z}_1^{\top} A \bm{z}_1 = \bm{z}_1^{\top} A(G[P]) \bm{z}_1 = \lambda_1(G[P])$. 
Similarly, $\bm{z}_2^{\top} A \bm{z} = \lambda_1(G[N])$. Moreover, 
$\bm{z}_1^{\top} A \bm{z}_2 = z_1(u_1) z_2(u_2)$. It follows that
\[
\begin{bmatrix}
\bm{z}_1^{\top} A\bm{z}_1 & \bm{z}_1^{\top} A\bm{z}_2 \\
\bm{z}_2^{\top} A\bm{z}_1 & \bm{z}_2^{\top} A\bm{z}_2
\end{bmatrix} = 
\begin{bmatrix}
\lambda_1(G[P]) & z_1(u_1)z_2(u_2) \\
z_1(u_1)z_2(u_2) & \lambda_1(G[N])
\end{bmatrix}.
\]
Its larger eigenvalue is at least
$\min\{\lambda_1(G[P]), \lambda_1(G[N])\} + z_1(u_1)z_2(u_2) > \lambda_2 + z_1(u_1)z_2(u_2)$. 
Then $\lambda_1 - \lambda_2 \geq z_1(u_1)z_2(u_2)$ by Lemma~\ref{lem:interlace-inequality}.

Now assume $|Z| = 1$. Then $\lambda_1(G[P]) = \lambda_1(G[N]) = \lambda_2$ by Lemma \ref{lem:zero-skeleton}. A direct calculation yields
\[
\begin{bmatrix}
\bm{z}_1^{\top} A\bm{z}_1 & \bm{z}_1^{\top} A\bm{e}_w & \bm{z}_1^{\top} A\bm{z}_2 \\
\bm{e}_w^{\top} A \bm{z}_1 & \bm{e}_w^{\top} A\bm{e}_w & \bm{e}_w^{\top} A \bm{z}_2 \\
\bm{z}_2^{\top} A \bm{z}_1 & \bm{z}_2^{\top} A \bm{e}_w & \bm{z}_2^{\top} A\bm{z}_2 
\end{bmatrix} =
\begin{bmatrix}
\lambda_2 & z_1(u_1) & 0 \\ z_1(u_1) & 0 & z_2(u_2) \\ 0 & z_2(u_2) & \lambda_2
\end{bmatrix}.
\]
Its characteristic polynomial is $(x - \lambda_2) (x(x - \lambda_2) - z_1(u_1)^2 - z_2(u_2)^2)$,
so its largest eigenvalue $\beta$ satisfies
$\beta-\lambda_2 = (z_1(u_1)^2 + z_2(u_2)^2)/\beta$.
Since $\lambda_1\geq\beta$, we obtain
$\lambda_1 - \lambda_2\geq\beta - \lambda_2
= (z_1(u_1)^2 + z_2(u_2)^2)/\beta\geq (z_1(u_1)^2 + z_2(u_2)^2)/\lambda_1$, as required.
\end{proof}

For the remainder of the paper, we set $q := \lceil \lambda_2\rceil$.
For each $i$, choose a vertex $w_i$ at which $\bm{z}_i$ attains its maximum, and let
$u_i=v_{i,0},v_{i,1},\ldots,v_{i,d_i}=w_i$
be a shortest $u_i$--$w_i$ path. Define the integer
$e_i:=h_i-d_i-q$, and $e = e_1+e_2$.

\begin{lemma}\label{lem:root-attenuation}
For $i=1,2$, let $\bm{z}_i$ be the unit Perron vector of $H_i$. Then $e_i\geq 0$ and
\begin{equation}\label{eq:root-attenuation}
z_i(u_i)\geq\frac{z_i(w_i)}{p_{d_i}(\lambda_1(H_i))}
\geq\frac{1}{\sqrt{h_i}\,\lambda_1^{d_i}}.
\end{equation}
\end{lemma}

\begin{proof}
For each $i\in\{1,2\}$, the eigenvalue\,--\,eigenvector equation at $w_i$ gives
$\lambda_1(H_i) \leq\deg_{H_i}(w_i)$. Since $\lambda_1(H_i) > q - 1$ from Lemma \ref{lem:petal-compression}, we have $\deg_{H_i}(w_i)\geq q$. If $d_i\geq1$,
no neighbor of $w_i$ other than $v_{i,d_i-1}$ lies on the chosen
shortest path. Hence $H_i$ has at least
$(d_i+1)+(q-1)=d_i+q$
vertices. If $d_i=0$, the degree bound $\deg_{H_i}(w_i)\geq \lambda_1(H_i)$ gives the estimate
$h_i\geq q+1$. Thus $e_i\geq 0$ in every case.

Now, we prove \eqref{eq:root-attenuation}. For $0\leq j<d_i$, define
\[
s_j:= \sum_{v\sim v_{i,j},\, v\notin\{v_{i,j-1}, v_{i, j+1}\}} z_i(v),
\]
where $v_{i,-1}$ is omitted when $j=0$. Let $b_j := z_i(v_{i,j})$, and set $b_{-1} := 0$.  
We claim that, for $1\leq k\leq d_i$,
\begin{equation}\label{eq:temporary-ineq-1}
b_k = p_k(\lambda_1(H_i)) z_i(u_i) - \sum_{j=0}^{k-1} p_{k-j-1}(\lambda_1(H_i)) s_j \quad (1\leq k\leq d_i),
\end{equation}
where $p_j$ is as defined in \eqref{eq:path-polynomials}. For $k=1$, this is clear.  If \eqref{eq:temporary-ineq-1} holds for $k$ and $k-1$, then
substitution in $b_{k+1}=\lambda_1(H_i) b_k - b_{k-1} - s_k$ and use of
$p_{\ell + 1} = \lambda_1(H_i) p_\ell-p_{\ell-1}$ gives the formula with $k+1$.
The claim therefore follows by induction.

Next, take $k=d_i$. Since $b_{d_i} = z_i(w_i)$, Equation \eqref{eq:temporary-ineq-1} gives
\begin{equation}\label{eq:source-continuant}
p_{d_i}(\lambda_1(H_i)) z_i(u_i) = z_i(w_i) + \sum_{j=0}^{d_i-1} p_{d_i-j-1}(\lambda_1(H_i)) s_j\geq z_i(w_i).
\end{equation}
$p_{d_i}(\lambda_1(H_i))\leq (\lambda_1(H_i))^{d_i}\leq\lambda_1^{d_i}$.
Since $1=\sum_vz_i(v)^2\leq h_i z_i(w_i)^2$, we have
$z_i(w_i)\geq h_i^{-1/2}$. Dividing \eqref{eq:source-continuant} by
$p_{d_i}(\lambda_1(H_i))$ proves both inequalities in
\eqref{eq:root-attenuation}.
\end{proof}

\begin{corollary}\label{coro:gap-lower-bound}
$\lambda_1 - \lambda_2 \geq 2\lambda_1^{-(n-2q-e)}/n$.
\end{corollary}

\begin{proof}
First assume $Z = \emptyset$, then $d_1+d_2=n-2q-e$. The assertion follows from
Lemma \ref{lem:petal-compression} and \eqref{eq:root-attenuation}. If $Z \neq\emptyset$, then 
$d_1+d_2+1=n-2q-e$, and Lemma \ref{lem:petal-compression} gives the same estimate (in fact with the
slightly better constant $4/(n-1)$).
\end{proof}

We end this section with locating $q$. Define
\begin{equation}\label{eq:F-def}
F_n(x) := (n-2x)\log x, \qquad
M_n := \max_{x\geq 3} F_n(x).
\end{equation}
Simple algebra shows that $F_n'(x) = n/x - 2\log x - 2$, and
$F_n''(x) = -n/x^2 - 2/x < 0$.
Thus $F_n(x)$ is strictly concave, and its only critical point is its unique
maximum. Hence, the unique maximizer $x_n$ of \eqref{eq:F-def} satisfies
\begin{equation}\label{eq:xn-equation}
 \frac{n}{2x_n}=\log x_n+1,
 \qquad
 x_n\sim\frac{n}{2\log n}.
\end{equation}

\begin{lemma}\label{lem:finite-defect}
With the notation above,
\begin{equation}\label{eq:q-scale}
q = x_n + O \Big(\sqrt{\frac{n}{\log n}}\Big) \sim\frac{n}{2\log n}.
\end{equation}
Moreover, $e_1 + e_2\leq 2$.
\end{lemma}

\begin{proof}
We begin by choosing an integer $x = x_n + O(1)$.
Since $F_n'(x_n)=0$, Taylor's theorem gives
\[
F_n(x) = F_n(x_n) + O\Big(\sup_{\xi} |F_n''(\xi)|\Big) = M_n + o(1).
\]
Consequently, $(n-2x-2)\log\sigma(x)=F_n(x)-2\log x+O(n/x^2) = M_n-O(\log n)$.

Set $\delta:=\gap(G)$ for short, since $\delta$ is the minimum gap among connected $n$-vertex graphs, 
Lemma \ref{lem:DK-gap-upper} gives
\[
\delta
\leq \gap\bigl(DK(x+1,n-2x-2)\bigr)
\leq \frac{8}{x+1}\sigma(x)^{-(n-2x-2)}.
\]
Taking logarithms gives $\log\delta\leq -\log(x+1) -(n-2x-2)\log\sigma(x) + O(1)$.
It follows that $\log\delta\leq -M_n + O(\log n)$.
On the other hand, we use $\lambda_1=\lambda_2+\delta\leq q+\delta$; together with 
Lemma~\ref{lem:magnitudes-order}, we have $\log\lambda_1\leq\log q+O(\delta/q)$.
Furthermore, since $e\geq 0$, Corollary \ref{coro:gap-lower-bound} implies
$\delta\geq 2\lambda_1^{-(n-2q)}/n$. Taking logarithms gives
$\log\delta\geq-F_n(q)-O(\log n)$.
Combining this with $\log\delta\leq -M_n + O(\log n)$ yields
\begin{equation}\label{eq:F-near-max}
 F_n(q)\geq M_n-O(\log n).
\end{equation}

Now, put $x=tx_n$, then a direct calculation gives
\[
 \frac{M_n-F_n(tx_n)}{n}
 =t-1-\log t+\frac{1-t+t\log t}{\log x_n+1}.
\]
The right-hand side is bounded away from zero when $t$ remains away from $1$, and it tends
to infinity as $t\downarrow 0$. Moreover, for $t$ outside a fixed compact interval,
the same conclusion follows directly from the monotonicity of $F_n$ on
either side of $x_n$. Since \eqref{eq:F-near-max} has error term of only
$O(\log n)=o(n)$, it follows that $q/x_n\to1$.  In particular,
$q,x_n\in[x_n/2,2x_n]$ for all sufficiently large $n$. Applying
Taylor's theorem between $x_n$ and $q$, there exists some $\xi$ between them such that
$F_n(q) = F_n(x_n) + F_n''(\xi)(q-x_n)^2/2$. It follows that 
\[
F_n(x_n) - F_n(q) = -\frac{1}{2} F_n''(\xi)(q-x_n)^2
= \Omega\Big(\frac{\log^2 n}{n}\Big) (q-x_n)^2.
\]
Combining this with \eqref{eq:F-near-max} yields $|q - x_n| =O (\sqrt{n/\log n})$,
proving \eqref{eq:q-scale}. 

We now apply Lemma~\ref{lem:DK-equations} with $r=q+1$. Since
$\log\bigl(q/\sigma(q)\bigr) = q^{-2} + O(q^{-4})$ and
$n/q^2 = o(1)$, we have
\[
\delta\leq \frac{8}{q+1}\sigma(q)^{-(n-2q-2)}\leq 9q^{-(n-2q-1)}.
\]
On the other hand, Corollary \ref{coro:gap-lower-bound} together with $\lambda_1\leq q+\delta$ gives
$\delta\geq \frac{2-o(1)}n q^{-(n-2q-e)}$.
Comparing the above inequalities, we obtain
$q^{e-1}\leq 5n$ for all sufficiently large $n$. If $e\geq 3$, this would imply
$q^2\leq 5n$, contradicting \eqref{eq:q-scale}. Hence $e\leq 2$.
\end{proof}

\section{Reducing the structure of extremal graphs}
\label{sec:terminal-blocks}

In this section, we complete the structural reduction.  
We first recall some results obtained in the preceding sections. If $G$ is a
minimizer, and $q=\lceil \lambda_2\rceil$, then
\begin{equation}\label{eq:terminal-standing-scale}
q\sim\frac{n}{2\log n}, \qquad 
\lambda_1(G)-\lambda _2(G) = o(q^{-K})~\text{for any fixed } K.
\end{equation}

\subsection{Two perturbation lemmas}

Let $G_i$ be connected graphs with $r_i\in V(G_i)$ for $i=1,2$.
Let $G_e$ denote the graph obtained from the disjoint union of $G_1$ and $G_2$ by
adding the edge $r_1r_2$. Let $G_v$ be the graph obtained from $G_e$ by subdividing edge $r_1r_2$ with $v$ the new vertex.

\begin{lemma}\label{lem:uniform-two-pole}
Let $G_i$ be connected graphs with $r_i\in V(G_i)$ for $i=1,2$. 
Let $\bm{p}_i$ be the unit Perron vectors of $G_i$. Suppose that
\[
\gamma := \min_{i=1,2} (\lambda_1(G_i) - \lambda_2(G_i))\geq 8, \quad \text{and} \quad 
|\lambda_1(G_1) - \lambda_1(G_2)|\leq\frac{\gamma}{4}.
\]
\begin{enumerate}
\item[$(1)$] If $\theta_+ \geq\theta_-$ are two eigenvalues of
$G_e$ lying above $\max_i \lambda_2(G_i) + 1$, then 
$\theta_+ - \theta_- \geq 2\,p_1(r_1)p_2(r_2) \big(1 - O(\gamma^{-2})\big)$.

\item[$(2)$] If $G_1$ and $G_2$ are isomorphic, then $\theta_+ - \theta_- \leq
2\,p_1(r_1)^2\big(1 + O(\gamma^{-2})\big)$.
\end{enumerate}
\end{lemma}

\begin{proof}
For $i\in\{1,2\}$, put $n_i:=|V(G_i)|$ and
$b_i:=\bm{p}_i^{\top}\bm{e}_{r_i}=p_i(r_i)$.
Choose $U_i\in\mathbb{R}^{n_i\times(n_i-1)}$ such that
$[\bm{p}_i\ U_i]$ is an orthogonal matrix. Then
$U_i^{\top}U_i=I_{n_i-1}$, $U_i^{\top}\bm{p}_i=\bm{0}$, and
$U_iU_i^{\top} = I_{n_i}-\bm{p}_i\bm{p}_i^{\top}$.
Let $\bm{q}_i:= U_iU_i^{\top}\bm{e}_{r_i}$ and
$\widehat{\bm{q}}_i:=U_i^{\top}\bm{e}_{r_i}$.
We have $\bm{q}_i=U_i\widehat{\bm{q}}_i$, and
$\bm{e}_{r_i}=b_i\bm{p}_i+\bm{q}_i$.
Consequently,
$\|\widehat{\bm{q}}_i\|^2=\|\bm{q}_i\|^2=1-b_i^2\leq 1$.

Define $B_i:=U_i^{\top}A(G_i)U_i$.
Since $A(G_i)$ is symmetric and
$A(G_i)\bm{p}_i=\lambda_1(G_i)\bm{p}_i$, we obtain
\[
[\bm{p}_i\ U_i]^{\top}A(G_i)[\bm{p}_i\ U_i]
=
\begin{bmatrix}
\lambda_1(G_i) & 0 \\
0 & B_i
\end{bmatrix}.
\]
In particular, the eigenvalues of $B_i$ are
$\lambda_2(G_i),\ldots,\lambda_{n_i}(G_i)$.

Order the vertices of $G_e$ so that those of $G_1$ precede
those of $G_2$. Then
\[
A:=A(G_e)=
\begin{bmatrix}
A(G_1) & \bm{e}_{r_1}\bm{e}_{r_2}^{\top} \\
\bm{e}_{r_2}\bm{e}_{r_1}^{\top} & A(G_2)
\end{bmatrix}.
\]
Let $S:=\left[
\begin{smallmatrix}
\bm{p}_1 & 0 & U_1 & 0 \\
0 & \bm{p}_2 & 0 & U_2
\end{smallmatrix}\right]$.
The matrix $S$ is orthogonal. By direct multiplication, we obtain
\begin{equation}\label{eq:two-pole-block-form}
S^{\top}AS=
\begin{bmatrix}
D & K^{\top} \\
K & C
\end{bmatrix},
\end{equation}
where
\[
D=
\begin{bmatrix}
\lambda_1(G_1) & b_1b_2 \\
b_1b_2 & \lambda_1(G_2)
\end{bmatrix},
\qquad
K=
\begin{bmatrix}
0 & b_2\widehat{\bm{q}}_1 \\
b_1\widehat{\bm{q}}_2 & 0
\end{bmatrix},
\]
and
\[
C=C_0+W, \qquad
C_0=
\begin{bmatrix}
B_1 & 0 \\
0 & B_2
\end{bmatrix},
\qquad
W=
\begin{bmatrix}
0 & \widehat{\bm{q}}_1\widehat{\bm{q}}_2^{\top} \\
\widehat{\bm{q}}_2\widehat{\bm{q}}_1^{\top} & 0
\end{bmatrix}.
\]

(1) The two columns of $K$ are orthogonal, and hence
$\|K\| = \max\{b_1\|\widehat{\bm{q}}_2\|, b_2\|\widehat{\bm{q}}_1\|\}
\leq 1$. Moreover, $\|W\| = \|\widehat{\bm{q}}_1\|\,\|\widehat{\bm{q}}_2\|\leq 1$.
Put
\[
\beta:= \max_{i=1,2} \lambda_2(G_i), \quad
\alpha_{\min}:= \min_{i=1,2} \lambda_1(G_i), \quad
\alpha_{\max}:= \max_{i=1,2} \lambda_1(G_i).
\]
Weyl's inequality gives $\lambda_1(C)\leq\lambda_1(C_0) + \|W\|\leq\beta + 1$.
On the other hand, The assumptions give $\lambda_2(G_i)\leq\lambda_1(G_i) - \gamma$, and hence 
\begin{equation}\label{eq:two-pole-beta}
\beta
\leq\alpha_{\max}-\gamma
\leq\alpha_{\min}-\frac{3\gamma}{4}.
\end{equation}
By Weyl's inequality applied to $D$, $\lambda_2(D)\geq\alpha_{\min} - b_1b_2
\geq\alpha_{\min}-1$. Cauchy interlacing for $D$ and $C$ yields
$\lambda_2(A)\geq\lambda_2(D)$, and $\lambda_3(A)\leq\lambda_1(C)$.
Thus at most two eigenvalues of $A$ can lie above $\lambda_1(C)$.
Since $\theta_+$ and $\theta_-$ both exceed $\beta+1$, they are the two largest eigenvalues of $A$. Consequently,
$\theta_-\geq\alpha_{\min}-1$.
Combining this with $\lambda_1(C)\leq\beta+1$ and \eqref{eq:two-pole-beta} gives
\begin{equation}\label{eq:two-pole-distance}
\theta_--\lambda_1(C)
\geq\alpha_{\min}-1-(\beta+1)
\geq\frac{3\gamma}{4}-2
\geq\frac{\gamma}{2},
\end{equation}
where the last inequality uses $\gamma\geq 8$. 

For $t\geq\theta_-$, define $R(t) = (tI - C)^{-1}$, and
$R_0(t) = (tI - C_0)^{-1}$.
It follows from \eqref{eq:two-pole-beta} and
\eqref{eq:two-pole-distance} that
\begin{equation}\label{eq:two-pole-resolvent-bounds}
\|R(t)\| = \frac{1}{t - \lambda_1(C)}\leq\frac{2}{\gamma}, \qquad
\|R_0(t)\| = \frac{1}{t - \lambda_1(C_0)}\leq\frac{2}{\gamma}.
\end{equation}
Since $C=C_0+W$, we see $tI - C_0 = tI - C + W$. Multiply on the right by $R(t)=(tI-C)^{-1}$, we have 
$(tI - C_0)R(t) = (tI-C)R(t) + WR(t) = I + WR(t)$. Multiply on the left by $R_0(t) = (tI - C_0)^{-1}$, then
$R(t) = R_0(t) + R_0(t)WR(t)$. Since $R_0(t)$ is block diagonal,
$\langle\bm{q}_2, R_0(t)\bm{q}_1\rangle = 0$.
It follows from \eqref{eq:two-pole-resolvent-bounds} that
\begin{equation}\label{eq:two-pole-cross-term}
\big|\langle\bm{q}_2, R(t)\bm{q}_1\rangle\big| =
\big|\langle\bm{q}_2, R_0(t)WR(t)\bm{q}_1\rangle\big|
\leq\frac{4}{\gamma^2}. 
\end{equation}
For $t\geq\theta_-$, define the $2\times2$ matrix $M(t):= D + K^{\top} R(t)K$.
Since $tI - C$ is invertible, the Schur-complement identity applied to
\eqref{eq:two-pole-block-form} gives $\det(tI-A) = \det(tI-C)\det(tI - M(t))$.
Thus, for $t>\lambda_1(C)$,
\begin{equation}\label{eq:A-Mt-same-spectra}
t\in\mathrm{spec} (A) \quad\Longleftrightarrow\quad
t\in\mathrm{spec} (M(t)).
\end{equation}
Let $\mu_+(t)\geq\mu_-(t)$ be the eigenvalues of $M(t)$. Since
$M'(t) = - K^{\top}(tI - C)^{-2}K\preccurlyeq 0$,
both $\mu_+(t)$ and $\mu_-(t)$ are nonincreasing functions of
$t$. Moreover, by $\|K\|\leq 1$ and \eqref{eq:two-pole-resolvent-bounds}, we deduce
$\|M'(t)\|\leq\|K\|^2\|R(t)\|^2\leq 4/\gamma^2$. Hence
\begin{equation}\label{eq:two-pole-Lipschitz}
\|M(t)-M(s)\|\leq\frac{4}{\gamma^2} |t-s| \qquad (t,s\geq\theta_-).
\end{equation}
Observe that each equation $t=\mu_+(t)$, $t=\mu_-(t)$
has at most one solution, because its left-hand side is strictly increasing
while its right-hand side is nonincreasing. 
In view of \eqref{eq:A-Mt-same-spectra} and the fact that $A$ has exactly two eigenvalues above $\lambda_1(C)$, 
$\theta_+ = \mu_+(\theta_+)$ and $\theta_- = \mu_-(\theta_-)$.

We now estimate the off-diagonal entry of $M(t)$. Note that
\[
M_{12}(t) = [1, 0] M(t) [0, 1]^{\top} =
b_1b_2 + b_1b_2\langle\bm{q}_2, R(t)\bm{q}_1\rangle
= b_1b_2 \left(1 + \langle\bm{q}_2, R(t)\bm{q}_1\rangle\right).
\]
It follows from \eqref{eq:two-pole-cross-term} that
\begin{equation}\label{eq:two-pole-off-diagonal-bounds}
b_1b_2\left(1-\frac{4}{\gamma^2}\right)
\leq |M_{12}(t)|
\leq
b_1b_2\left(1+\frac{4}{\gamma^2}\right).
\end{equation}
Observe that for any real symmetric $2\times2$ matrix,
the difference between its two eigenvalues is at least two times its $(1,2)$-entry in absolute value.
Applying this to $M(\theta_-)$ gives
\begin{equation}\label{eq:two-pole-effective-gap}
\mu_+(\theta_-)-\mu_-(\theta_-)
\geq
2b_1b_2\left(1-\frac{4}{\gamma^2}\right).
\end{equation}
Let $\Delta = \theta_+ - \theta_-$. By \eqref{eq:two-pole-Lipschitz},
$\mu_+(\theta_+)\geq\mu_+(\theta_-) - 4\Delta/\gamma^2$.
Using $\theta_+ = \mu_+(\theta_+)$ and $\theta_- = \mu_-(\theta_-)$, we get
\[
\theta_+ = \mu_+(\theta_+)\geq
\mu_+(\theta_-)-\frac{4}{\gamma^2}\Delta
= \theta_- + \bigl(\mu_+(\theta_-)-\mu_-(\theta_-)\bigr)
 -\frac{4}{\gamma^2}\Delta.
\]
Therefore $(1 + 4/\gamma^2)\Delta\geq\mu_+(\theta_-) - \mu_-(\theta_-)$.
Together with \eqref{eq:two-pole-effective-gap}, this yields
\[
\theta_+-\theta_-
\geq
2b_1b_2
\frac{1-4\gamma^{-2}}{1+4\gamma^{-2}}
=
2b_1b_2\bigl(1-O(\gamma^{-2})\bigr),
\]
which proves $(1)$.

(2) By identifying $G_1$ and $G_2$ via an isomorphism if necessary, we may
assume that $A(G_1) = A(G_2)$ and $\bm{e}_1 = \bm{e}_2$. Consequently,
$\lambda_1(G_1) = \lambda_1(G_2)$, $b_1 = b_2 =:b$, $B_1 = B_2$, and
$\bm{q}_1 = \bm{q}_2$. Let $J =
\left[\begin{smallmatrix}
0 & 1 \\
1 & 0
\end{smallmatrix}\right]$,
$\widehat{J} =
\left[\begin{smallmatrix}
0 & I \\
I & 0
\end{smallmatrix}\right]$.
It follows that $JD = DJ$, $\widehat{J} C = C\widehat{J}$, $\widehat{J} K = KJ$,
which implies that $JM(t) = M(t)J$. Therefore $M(t)$ has the form
\[
M(t)=
\begin{bmatrix}
a(t)&h(t)\\
h(t)&a(t)
\end{bmatrix},
\]
and hence $\mu_+(t) - \mu_-(t) = 2|h(t)| = 2|M_{12}(t)|$.
Since $\mu_+$ is nonincreasing,
\[
\theta_+ - \theta_- = \mu_+(\theta_+) - \mu_-(\theta_-)
\leq\mu_+(\theta_-) - \mu_-(\theta_-) = 2|M_{12}(\theta_-)|.
\]
The upper bound in \eqref{eq:two-pole-off-diagonal-bounds} now gives
$\theta_+ - \theta_-\leq 2b^2 (1 + 4/\gamma^2) = 2b_1^2\big(1 + O(\gamma^{-2})\big)$.
This finishes the proof of (2).
\end{proof}

For the case of two petals joined at a single vertex (i.e., $|Z| = 1$), 
we need the following result.

\begin{lemma}\label{lem:zero-parity}
Let $G_i$ be connected graphs with $r_i\in V(G_i)$ for $i=1,2$. 
Let $\gamma$ and $\bm{p}_i$ be as Lemma \ref{lem:uniform-two-pole}. Suppose that
$\lambda_1(G_1) = \lambda_1(G_2) = \theta$, 
$\theta - \max_i\lambda_2(G_i)\geq\gamma$, and 
$\theta = \Theta(\gamma)\to\infty$.
Then $\theta$ is the second largest eigenvalue of $G_v$, and
\begin{equation}\label{eq:zero-parity-gap}
\lambda_1(G_v) - \lambda_2(G_v) = \frac{p_1(r_1)^2 + p_2(r_2)^2}{\theta}
\left(1 + O(\theta^{-2})\right).
\end{equation}
\end{lemma}

\begin{proof}
For each $i\in\{1,2\}$, let $b_i:=p_i(r_i)$, $A_i:=A(G_i)$, and let $\bm{p}_i$ be a unit Perron vector of
$A_i$, so that $A_i \bm{p}_i = \theta \bm{p}_i$. 
Since $G_i$ is connected, the Perron--Frobenius theorem implies that
$\theta$ is a simple eigenvalue of $A_i$, $\bm{p}_i$ can be chosen strictly
positive, and hence $b_i>0$. 

Deleting the vertex $v$ from $G_v$ leaves the disjoint union $G_1\sqcup G_2$,
whose two largest eigenvalues are both equal to $\theta$. By Cauchy's interlacing theorem,
we have $\lambda_1(G_v)\geq \theta\geq\lambda_2(G_v)\geq\theta$.
Consequently, $\lambda_2(G_v)=\theta$,
and at most one eigenvalue of $G_v$ lies strictly above $\theta$.

With the vertices ordered as $V(G_1),v,V(G_2)$, the adjacency matrix of
$G_v$ is
\[
A(G_v) =
\begin{bmatrix}
A_1 & \bm{e}_{r_1} & 0 \\
\bm{e}_{r_1}^{\top} & 0 & \bm{e}_{r_2}^{\top} \\
0 & \bm{e}_{r_2} & A_2
\end{bmatrix}.
\]
Consider the vector $\bm{z}$ with $\bm{z}^{\top} = [b_2\bm{p}_1^{\top}, 0, -b_1\bm{p}_2^{\top}]$.
On $G_1$ and $G_2$, respectively, we have
$A_1(b_2 \bm{p}_1) = \theta b_2 \bm{p}_1$, and $A_2(-b_1 \bm{p}_2) = -\theta b_1 \bm{p}_2$.
At the vertex $v$, the corresponding coordinate of $A(G_v)\bm{z}$ is
$\bm{e}_{r_1}^{\top} (b_2\bm{p}_1) + \bm{e}_{r_2}^{\top}(-b_1\bm{p}_2)
= b_2b_1 - b_1b_2 = 0$. Thus $\bm{z}$ is an eigenvector of $G_v$ with eigenvalue $\theta$.
Moreover, $G_v$ is connected, whereas the $\theta$-eigenvector $\bm{z}$
changes sign and vanishes at $v$. It therefore cannot be a Perron
vector of $G_v$. The Perron--Frobenius theorem then shows that
$\lambda_1(G_v) > \theta$. Hence there is exactly one eigenvalue above
$\theta$.

We next determine this eigenvalue. For $x>\theta$, let
$g_i(x) = \bm{e}_{r_i}^{\top} (xI - A_i)^{-1} \bm{e}_{r_i}$.
Suppose that $x>\theta$ is an eigenvalue of $G_v$, with eigenvector
$\bm{w}$, where $\bm{w}^{\top} = [\bm{w}_1^{\top}, \alpha, \bm{w}_2^{\top}]$.
The eigenvalue--eigenvector equation on $G_i$ is
$(xI - A_i)\bm{w}_i = \alpha \bm{e}_{r_i}$.
Since $x > \lambda_1(A_i) = \theta$, the matrix $xI - A_i$ is invertible, and hence
$\bm{w}_i = \alpha(xI - A_i)^{-1} \bm{e}_{r_i}$.
Moreover, $\alpha\neq 0$, since otherwise the preceding equations would
give $\bm{w}_1 = \bm{w}_2 = 0$. The eigenvalue\,--\,eigenvector equation at $v$ now becomes
\[
x\alpha = \bm{e}_{r_1}^{\top} \bm{w}_1 + \bm{e}_{r_2}^{\top} \bm{w}_2
= \alpha\big( g_1(x) + g_2(x)\big).
\]
Dividing by $\alpha$ gives the equation $x = g_1(x) + g_2(x)$.
This equation has exactly one solution above $\theta$.
Indeed, $g_i'(x) = -\bm{e}_{r_i}^{\top} (xI - A_i)^{-2} \bm{e}_{r_i} < 0$,
so $g_1(x) + g_2(x) - x$ is strictly decreasing. Furthermore, by the spectral decomposition of $A_i$, 
$g_i(x)\sim\frac{b_i^2}{x-\theta}$ as $x\downarrow\theta$,
and hence $g_1(x)+g_2(x)-x\to+\infty$ as $x\downarrow\theta$. On the
other hand, $g_i(x)=O(x^{-1})$ as $x\to\infty$, so
$g_1(x)+g_2(x)-x\to-\infty$. Thus the solution is unique and must equal
$\lambda_1(G_v)$.

Set $\lambda_1(G_v) = \theta + \delta$, where $\delta > 0$.
Let $\mu_{ij}$ and $\bm{z}_{ij}$ be the eigenvalues and an orthonormal eigenbasis of $A_i$, with
$\mu_{i1} = \theta$, $\bm{z}_{i1} = \bm{z}_i$, and $\mu_{ij}\leq\lambda_2(G_i)\leq\theta - \gamma$
($j\geq 2$). The spectral decomposition for $A_i$ gives
\[
g_i(\theta+\delta) = \sum_{j\geq 1} \frac{\left|\langle \bm{e}_{r_i}, \bm{z}_{ij}\rangle\right|^2}{\theta+\delta-\mu_{ij}}
= \frac{b_i^2}{\delta} + \sum_{j\geq 2} \frac{\left|\langle \bm{e}_{r_i}, \bm{z}_{ij}\rangle\right|^2}{\theta+\delta-\mu_{ij}}.
\]
Define the second sum to be $h_i(\delta)$. All its summands are
nonnegative, and, for $j\geq 2$,
$\theta + \delta - \mu_{ij}\geq\gamma + \delta\geq\gamma$.
Since $\sum_{j\geq 2} \left|\langle \bm{e}_{r_i}, \bm{z}_{ij}\rangle\right|^2 = 1 - b_i^2\leq 1$,
we obtain the slightly sharper estimate
$0\leq h_i(\delta)\leq 1/\gamma$. Thus $g_i(\theta + \delta) = b_i^2/\delta + h_i(\delta)$.
It follows that $\theta + \delta = (b_1^2 + b_2^2)/\delta + h_1(\delta) + h_2(\delta)$,
or equivalently 
\begin{equation}\label{eq:temporary-1}
\delta\big(\theta + \delta - h_1(\delta) - h_2(\delta)\big) = b_1^2 + b_2^2.
\end{equation}
In light of $\gamma = \Theta(\theta)$ and $h_i(\delta) \leq 1/\gamma$, we see
$\theta + \delta - h_1(\delta) - h_2(\delta)\geq\theta - 2/\gamma
\geq\theta/2$ for all sufficiently large $\theta$. 
It follows from \eqref{eq:temporary-1} that $0 < \delta\leq 2(b_1^2 + b_2^2)/\theta
= O(\theta^{-1})$.
Together with $h_i(\delta)\leq\gamma^{-1}=O(\theta^{-1})$, this yields
$\theta + \delta - h_1(\delta) - h_2(\delta) = \theta + O(\theta^{-1})
= \theta\left(1 + O(\theta^{-2})\right)$. Consequently, we obtain
\[
\delta = \frac{b_1^2+b_2^2}{\theta\left(1+O(\theta^{-2})\right)}
= \frac{b_1^2+b_2^2}{\theta} \left(1+O(\theta^{-2})\right).
\]
Together with $\lambda_2(G_v)=\theta$ and
$\lambda_1(G_v) = \theta+\delta$, this completes the proof.
\end{proof}

\subsection{Rough structure}

In this subsection, we describe the coarse structure of $G$. To do this we give a precise formula for $\gap(G)$ in terms of eigenvector entries and then show that this implies $G$ must have structure quite similar to a double kite. 

\begin{theorem}\label{thm:terminal-extraction}
There are two disjoint vertex sets $V_1\subseteq V(H_1)$, $V_2\subseteq V(H_2)$ 
with $w_i\in V_i$, such that the induced subgraphs
$B_i:=H_i[V_i]$ have the following properties.

\begin{enumerate}
\item[\textnormal{(1)}] For each $i\in\{1,2\}$, $B_i$ is connected, and either $B_i=H_i$ or
$B_i\subsetneq H_i$. In the latter case, there are vertices 
$\widehat{u}_i\in V(B_i)$, $\widehat{v}_i\in V(H_i)\setminus V(B_i)$ such that
$E_{H_i} (V(B_i), V(H_i)\setminus V(B_i)) = \{\widehat u_i \widehat{v}_i\}$. 

\item[\textnormal{(2)}]
$q\leq |V(B_i)|\leq q + 8$. Moreover, $\operatorname{dist}_{B_i}(\widehat u_i,w_i)\leq 6$ and
$\max_{v\in V(B_i)} \operatorname{dist}_{B_i}(\widehat u_i,v)\leq 11$.

\item[\textnormal{(3)}]
The induced subgraph $C:= G[V(G)\setminus (V(B_1)\cup V(B_2))]$
is connected, $\widehat{v}_1, \widehat{v}_2\in V(C)$, and
\[
E_G(V(B_1)\cup V(B_2), V(C)) = \{\widehat{u}_1\widehat{v}_1, \widehat{u}_2\widehat{v}_2\}.
\]
Furthermore, $|V(C)| = n-m_1-m_2 = \Theta(q\log q)$ and $\Delta(C)\leq 4$.
\end{enumerate}
\end{theorem}

\begin{proof}
We first recall the notation used in the previous section. 
Let $H_1:=G[P]$ and $H_2:=G[N]$, and let $\bm{z}_i$ be the unit Perron vector 
of $H_i$ for $i\in\{1,2\}$.
When $Z=\emptyset$, the edge $u_1u_2$ is the unique edge joining $H_1$ and $H_2$, 
where $u_i\in V(H_i)$. When $|Z|=1$, let $w$ denote the unique vertex of $Z$; 
in this case, $wu_i$ is the unique edge joining $w$ to $H_i$. 
Finally, for each $i\in\{1,2\}$, let $w_i\in V(H_i)$ be a vertex at which 
$\bm{z}_i$ attains its maximum coordinate.

Fix $i\in\{1,2\}$. To simplify notation, we suppress the index, write
$H := H_i$, $u:=u_i$, $w:=w_i$ and $d:=d_i$, and again use $\alpha_i = \lambda_1(H_i)$. Let $P$ denote a shortest path $u = v_0v_1\cdots v_d=w$,
and $R:=V(H)\setminus V(P)$, $e_i=e^\ast$. Define $L := N_H(w)\cap R$, $F := R\setminus L$. Recall that $e* = |V(H)| - d - q$.

The eigenvalue\,--\,eigenvector equation at $w$ gives
$\alpha_i\leq d_H(w)$. Since $\alpha_i > q - 1$ by Lemma \ref{lem:petal-compression}, we have
$d_H(w)\geq q$. At most one neighbor of $w$ lies on $P$, so
$|L|\geq q-1$. Because $|R|=q+e^\ast-1$, it follows that $|F|\leq e^\ast$. 

We now construct vertex set $S\subseteq V(H)$.  Put $j=d$ and initially set
\begin{equation}
S = L\cup\{v_j,v_{j+1},\ldots,v_d\}.
 \label{eq:mgd-unified-initial-set}
\end{equation}
Clearly, the induced subgraph $H[S]$ is connected. 
Starting from \eqref{eq:mgd-unified-initial-set}, repeatedly perform
the following operations:

(i). if a vertex $v\in F\setminus S$ has a neighbor in $S$, add $v$ to $S$;

(ii). if a vertex $v\in S\setminus V(P)$ is adjacent to $v_k$
for some $k<j$, replace $j$ by the least such $k$ and add
all the vertices $v_k,v_{k+1},\ldots,v_d$ to $S$.

Note that vertices in operation~(i) may be added one at a time.
Since at most $|F|\leq e^\ast$ vertices can be added by
operation~(i), and $j$ can only decrease. Hence the
procedure terminates. We next bound how far $j$ can move toward $v_0$.

\begin{claim}\label{claim:d-j}
If $r$ vertices of $F$ have been
added, then $d-j\leq 2r + 2$, and every added vertex of $F$ can be joined to $w$, by a path
contained in the current induced subgraph $H[S]$, of length at most
$2r + 1$.
\end{claim}

\begin{proof}[Proof of Claim \ref{claim:d-j}]
Before any vertex of $F$ has been added, every off-path vertex in
$S$ belongs to $L$ and is adjacent to $w$. If $v\in L$ is
also adjacent to $v_k$, then $v_kvw$
is an $v_k$--$w$ path of length two.  
Since $P$ is a shortest $u$--$w$ path, $d-k=\dist_H(v_k,w)\leq 2$.
Therefore, before any vertex of $F$ is added, $d-j\leq2$,
which proves the assertion for $r=0$.

Suppose now that $r-1$ vertices of $F$ have been added and that
the assertions hold. Let $v\in F$ be the next vertex added by
operation~(i), and let $y\in S$ be a neighbor of $v$.

If $y\in L$, then $y\sim w$, so $v$ has a path to $w$ of
length two. If $y$ is one of the previously added vertices of
$F$, then the induction hypothesis gives a $y$--$w$ path of
length at most $2(r-1)+1=2r-1$.
After adding the edge $vy$, we obtain a $v$--$w$ path of length
at most $2r$. If $y\in V(P)$, before $v$ is added, the induction hypothesis gives
$d-j\leq2(r-1)+2=2r$.
Thus the subpath from $y$ to $w$ has length at most $2r$, and
the edge $vy$, followed by this subpath, gives a $v$--$w$ path
of length at most $2r+1$.  

Suppose operation~(ii) is now triggered by an edge
$v_kz$, where $z\in S\setminus V(P)$. If $z\in L$, then
$z\sim w$. If $z\in F$, then there is a $z$--$w$ path of
length at most $2r+1$. In either case, the edge $v_kz$, followed
by a path from $z$ to $w$, gives a $v_k$--$w$ walk of length
at most $2r+2$. Therefore $d-k=\dist_H(v_k,w)\leq 2r + 2$.
After replacing $j$ by $k$, we retain $d-j\leq 2r+2$.
This completes the induction.
\end{proof}

Let $S$ and $j$ denote the set and the index when the procedure
terminates. Hence, $d-j\leq2e^\ast+2\leq 6$. In the following we will prove that $S$ is the desired set $V_i$.
Define $B_i=H_i[S]$, $\widehat{u}_i=v_j$.

(1) We first determine the edges between $S$ and $V(H_i)\setminus S$.
Suppose $j>0$. There is no edge from $S$ to a vertex of
$F\setminus S$, since otherwise operation~(i) would
still apply. There is also no vertex of $L$ outside $S$, since all
vertices of $L$ were included initially. Thus no edge from $S$
to $V(H_i)\setminus S$ can have its endpoint outside $S$ in
$R=L\cup F$. There is also no edge from a vertex of $S\setminus V(P)$ to a path
vertex $v_k$ with $k<j$, because otherwise operation
(ii) would still apply.
It remains to consider an edge $v_av_b\in E(H_i)$, $b<j\leq a$.
If this edge is not $v_{j-1}v_j$, then $a-b\geq 2$. The edge
$v_bv_a$ would then replace the subpath
$v_bv_{b+1}\cdots v_a$ of length $a-b\geq 2$ by a path of length one. This contradicts the
fact that $P$ is a shortest path. Therefore
$E_{H_i} (S, V(H_i)\setminus S) = \{v_{j-1}v_j\}$.

Suppose instead that $j=0$. Then $V(P)\subseteq S$. If
$S\neq V(H_i)$, the connectedness of $H_i$ gives an edge between $S$
and $V(H_i)\setminus S$. Every vertex outside $S$ belongs to
$F$, because $V(P)\cup L\subseteq S$. Such an edge would permit
another application of operation~(i), contrary to termination. Hence $S=V(H_i)$.
In this case, $B_i=H_i$, $\widehat{u}_i = u_i$,
and the edge by which $H_i$ is attached to the remainder of $G$
is the unique edge joining $B_i$ to $G-V(B_i)$.
Thus $B_i$ is a connected induced subgraph containing $w_i$, and exactly one edge joins $B_i$ to
the remainder of $G$.

(2) We next estimate $|V(B_i)|$. By $|L|\geq q-1$, $B_i$ contains $w_i$ and at least
$q-1$ vertices of $L$. Therefore $|V(B_i)|\geq 1 + |L|\geq q$.
The subgraph $B_i$ contains at most all $|R| = q + e^\ast - 1$
off-path vertices and exactly $d-j+1$ vertices of $P$. It follows from $d-j\leq 6$ that
\[
|V(B_i)| \leq (q + e^\ast - 1) + (d - j + 1)\leq q + 8.
\]

We now prove the distance estimates. Since the $\widehat{u}_i$--$w_i$ subpath of $P$ belongs to $B_i$,
$\dist_{B_i}(\widehat{u}_i, w_i)=d-j$. Equation $d-j\leq 6$ therefore gives
$\dist_{B_i} (\widehat{u}_i, w_i)\leq 6$.
Every vertex of $L$ is adjacent to $w_i$, and every included vertex of $F$ has a path to $w_i$ of length at most
$2e^\ast + 1\leq 5$. A path vertex in $B_i$ is at distance at most $d-j\leq 6$ from
$\widehat{u}_i$. A vertex of $L$ is at distance at most
$(d-j)+1\leq 7$ from $\widehat{u}_i$. Finally, a vertex of $F\cap V(B_i)$ is at
distance at most $(d-j)+(2e^\ast+1)\leq 6+5=11$ from $\widehat{u}_i$. Hence
$\max_{v\in V(B_i)}\dist_{B_i}(\widehat u_i,v)\leq 11$.

(3) It remains to study the induced subgraph
$C = G[V(G)\setminus(V(B_1)\cup V(B_2))]$.
Every vertex of $C$ outside the remaining portions of $P_1$ and
$P_2$ belongs to one of the sets $F_1,F_2$. Their total number is
at most $|F_1|+|F_2|\leq e_1+e_2\leq 2$.
A vertex on the remaining central path has at most two neighbors on
that path and at most $2$ neighbors outside it. Thus its degree in
$C$ is at most $4$.

Let $v$ be a vertex of $C$ outside the selected path.
The vertex $v$ can be adjacent to at most three consecutive vertices
of the selected path in its subgraph $H_i$. Indeed, if
$v\sim v_a$, $v\sim v_b$, $b-a\geq 3$,
then $v_avv_b$ is an $v_a$--$v_b$ path of length two, whereas
the corresponding subpath of $P_i$ has length $b-a\geq 3$. This
contradicts the fact that $P_i$ is a shortest path. The vertex $v$
also has at most one neighbor outside the selected paths, because
there are at most two such vertices in total. Therefore
$d_C(v)\leq 3+1=4$.

Finally, the bounds proved above give $|V(B_i)|=q+O(1)$.
Since $B_1,B_2,C$ have pairwise disjoint vertex sets whose union is
$V(G)$, we have $|V(C)| = n-|V(B_1)|-|V(B_2)| = n-2q+O(1)$.
Using the previously established estimate
$n-2q=\Theta(q\log q)$, we conclude that $|V(C)|=\Theta(q\log q)$.
\end{proof}

In the following we always assume $\bm{z}_i$ is the unit Perron vector of $B_i$ for $i\in\{1,2\}$, and
$\widehat{u}_i\in B_i$, $\widehat{v}_i\in C$ and $\widehat{u}_i\widehat{v}_i$
is the unique edge joining $B_i$ to $C$.

\begin{lemma}\label{lem:terminal-parameters}
For $i=1,2$, we have $\lambda_1(B_i) = q + O(1)$, $\lambda_1(B_i) - \lambda_2(B_i) = \Theta(q)$, and 
$z_i(\widehat{u}_i)\geq q^{-8}$.
Furthermore, every positive eigenvalue of $G$ other than 
$\lambda_1$ and $\lambda_2$ is $O(\sqrt{q})$.
\end{lemma}

\begin{proof}
If $Z = \emptyset$, then $\lambda_2<\lambda_1(H_i)\leq\lambda_1$ by Lemma \ref{lem:petal-compression}, 
while if $Z \neq \emptyset$, then $\lambda_1(H_i) = \lambda_2$. Since $q-1<\lambda_2\leq q$ and
$\lambda_1-\lambda_2=o(1)$, in both cases $\lambda_1(H_i) = q + O(1)$.

If $B_i = H_i$, we clearly have $\lambda_1(B_i) = q + O(1)$. If $B_i\subseteq H_i$, Theorem \ref{thm:terminal-extraction} implies that
$\widehat{v}_i$ is the unique neighbor of $\widehat{u}_i$ outside $B_i$. Restrict $\bm{z}_i$ to $B_i$. The only lost edge is $uu'$, so
\begin{equation}\label{eq:terminal-restriction-rayleigh}
(\bm{z}_i|_{B_i})^{\top} A(B_i) \bm{z}_i|_{B_i}
= \lambda_1(H_i)\|\bm{z}_i|_{B_i}\|^2 - z_i(\widehat{u}_i) z_i(\widehat{v}_i).
\end{equation}
Indeed, for every $v\in V(B_i)$, the eigenvalue\,--\,eigenvector equation at $v$ is
\[
\lambda_1(H_i) z_i(v) = \sum_{w\in V(B_i),\,w\sim v} z_i(w) + \bm{1}_{\{v=u\}} z_i(\widehat{v}_i).
\]
Multiplying by $z_i(v)$ and summing over $v\in V(B)$ gives
\begin{align*}
\lambda_1(H_i)\sum_{v\in V(B_i)} z_i(v)^2
& = \sum_{v\in V(B_i)} \sum_{w\in V(B_i),\,w\sim v} z_i(v)z_i(w) + z_i(\widehat{u}_i) z_i(\widehat{v}_i) \\
& = (\bm{z}_i|_{B_i})^{\top} A(B_i) \bm{z}_i|_{B_i} + z_i(\widehat{u}_i) z_i(\widehat{v}_i).
\end{align*}
so \eqref{eq:terminal-restriction-rayleigh} follows. By the Rayleigh principle,
\[
\lambda_1(B_i)\geq
\frac{(\bm{z}_i|_{B_i})^{\top} A(B_i)(\bm{z}_i|_{B_i})}{\|\bm{z}_i|_{B_i}\|^2}
= \lambda_1(H_i) - \frac{z_i(\widehat{u}_i)z_i(\widehat{v}_i)}{\|\bm{z}_i|_{B_i}\|^2}.
\]
Since $z_i(\widehat{u}_i),z_i(\widehat{v}_i)\leq z_i(w_i)$ and $\|\bm{z}_i|_{B_i}\|^2 = \Omega(q z_i(w_i)^2)$,
we get $\lambda_1(H_i)\geq\lambda_1(B_i)\geq\lambda_1(H_i) - O(q^{-1})$. Hence, $\lambda_1(B_i) = q+O(1)$.

Next, since $|V(B_i)| = q + O(1)$, and $\mathrm{tr}\, A(B_i)^2 = 2|E(B_i)|\leq |V(B_i)|(|V(B_i)| - 1)$, we have
$\lambda_2(B_i)^2 \leq \mathrm{tr}\, A(B_i)^2 - \lambda_1(B_i)^2 = O(q)$.
Thus $\lambda_1(B_i) - \lambda_2(B_i)=\Theta(q)$.

Theorem \ref{thm:terminal-extraction} gives every vertex of $B_i$ is at distance at most $11$ from $\widehat u_i$.
If $v$ is a vertex in $B_i$ at which $\bm{z}_i$ is maximum, repeated use of
the eigenvalue\,--\,eigenvector equations along a shortest $\widehat{u}_i$\,--\,$v$ path gives
$z_i(\widehat{u}_i)\geq \lambda_1(B_i)^{-6}\widehat{z}_i(v) \geq \lambda_1(B_i)^{-6} |V(B_i)|^{-1/2}\geq q^{-8}$
for all sufficiently large $q$.

Finally, consider the graph $B_1\sqcup C\sqcup B_2$. Its third-largest
positive eigenvalue is $O(\sqrt{q})$. Adding edges $\widehat{u}_1\widehat{v}_1$ and $\widehat{u}_2\widehat{v}_2$ 
to $B_1\sqcup C\sqcup B_2$ changes the operator norm by at most $2$, and the desired assertion then 
follows from Weyl's inequality.
\end{proof}

\subsection{The induced subgraph $G[V(C)]$ is a path}

With the notation in Theorem \ref{thm:terminal-extraction},
by Theorem \ref{thm:terminal-extraction}, $G$ is obtained from the disjoint union
$B_1\sqcup C\sqcup B_2$
by adding precisely the two edges $\widehat{u}_1\widehat{v}_1$ and $\widehat{u}_2\widehat{v}_2$,
where $\widehat{v}_1$, $\widehat{v}_2\in V(C)$ and $\widehat{u}_i\in V(B_i)$, $i=1,2$.
Let $s:= |V(C)|$. In this subsection we will see that $V(C)$ induces a path. The proof is technical so we give a high-level outline now. We will give a formula for $\gap(G)$ in terms of the spectral radii of the graphs induced by $B_1$ and $B_2$ and the resolvent of $A(C)$ (see \eqref{eq:loaded-two-pole-formula}). We will compare the spectral gap of $G$ with the spectral gap of the graph given by two copies of either $G_1$ or $G_2$ joined by a path. We will show that the error terms can be controlled enough that the term in \eqref{eq:loaded-two-pole-formula} coming from the resolvent of $A(C)$ will be the determining factor in which of these graphs has the smallest spectral gap. Finally, we will estimate this term and see that it is smallest if $C$ is a path. This final estimate uses Lemma \ref{lem:corridor-transfer} which gives a lower bound on the term when $C$ is not a path, and which we prove next. 

\begin{lemma}\label{lem:corridor-transfer}
Let $x\in [q-K, q+K]$. If $C$ is not $P_s$, then $R_C(x)_{\widehat{v}_1\widehat{v}_2} \geq x^{-(s-1)}$. 
Moreover, $R_C(x)_{\widehat{v}_1\widehat{v}_1} = x^{-1} + O(x^{-3})$, and
$|R_C'(x)_{\widehat{v}_1\widehat{v}_1}| = \|R_C(x) \bm{e}_{\widehat{v}_1}\|^2 \leq (x-4)^{-2} = O(q^{-2})$,
and the same estimates hold at $\widehat{v}_2$.
\end{lemma}

\begin{proof}
By Theorem \ref{thm:terminal-extraction}, $\Delta(C) \leq 4$. Since $\lambda_1(C)\leq\Delta(C)\leq 4$, for $x>4$, we have
\begin{equation}\label{eq:corridor-walk-expansion}
R_C(x) = x^{-1}\sum_{j\geq 0} \frac{A(C)^j}{x^j}.
\end{equation}
If $\operatorname{dist}_C(\widehat{v}_1, \widehat{v}_2)=s-1$, then $C = P_s$ with
ends $\widehat{v}_1, \widehat{v}_2$. Hence, if $C$ is not this path, then
$\operatorname{dist}_C(\widehat{v}_1, \widehat{v}_2)\leq s-2$. The term in
\eqref{eq:corridor-walk-expansion} corresponding to one shortest walk gives
$R_C(x)_{\widehat{v}_1\widehat{v}_2} \geq x^{-(s-1)}$.

Since there is one closed walk of length zero, no closed walk of length one, and at most $4^j$
walks of length $j$ starting at $\widehat{v}_1$, we have 
\begin{align*}
R_C(x)_{\widehat{v}_1\widehat{v}_1} = \bm{e}_{\widehat{v}_1}^{\top} R_C(x) \bm{e}_{\widehat{v}_1} 
& = \frac{1}{x} \bigg (1 + \sum_{j\geq 2} \frac{\bm{e}_{\widehat{v}_1}^{\top} A(C)^j \bm{e}_{\widehat{v}_1}}{x^j}\bigg) \\
& \leq \frac{1}{x} \bigg(1 + \sum_{j\geq 2} \Big(\frac{4}{x}\Big)^j\bigg)
= \frac{1}{x} + O\Big(\frac{1}{x^3}\Big).
\end{align*}

Finally, since $R_C'(x) = -R_C(x)^2$, we obtain
$|R_C'(x)_{\widehat{v}_1\widehat{v}_1}| = \|R_C(x) \bm{e}_{\widehat{v}_1}\|^2
\leq \|R_C(x)\|^2 \leq (x-4)^{-2}$.
\end{proof}

In the following, we write
\[
r_{1,C}(x) := R_C(x)_{\widehat{v}_1\widehat{v}_1}, \qquad r_{2,C}(x) := R_C(x)_{\widehat{v}_2\widehat{v}_2}, \qquad 
\tau_C(x) := R_C(x)_{\widehat{v}_1\widehat{v}_2},
\]
and $Q_{B_i, \widehat{u}_i}(x) := \phi(B_i, x)/\phi(B_i - \widehat{u}_i,x)$, 
$F_{i,C}(x) := Q_{B_i,\widehat{u}_i}(x) - r_{i,C}(x)$.

With the vertices ordered as $V(B_1)$, $V(B_2)$, $V(C)$,
the matrix $xI - A(G)$ has block form
\[
xI - A(G) =
\begin{bmatrix}
xI - A(B_1) & 0 & -\bm{e}_{\widehat{u}_1} \bm{e}_{\widehat{v}_1}^{\top} \\[1mm]
0 & xI - A(B_2) & -\bm{e}_{\widehat{u}_2} \bm{e}_{\widehat{v}_2}^{\top} \\[1mm]
-\bm{e}_{\widehat{v}_1} \bm{e}_{\widehat{u}_1}^{\top} & -\bm{e}_{\widehat{v}_2} \bm{e}_{\widehat{u}_2}^{\top} & xI - A(C)
\end{bmatrix}.
\]
Suppose that $x\notin\operatorname{spec}(B_1)\cup\operatorname{spec}(B_2)\cup
\operatorname{spec}(C)$. For $i=1,2$, define
$g_i(x) := \bm{e}_{\widehat{u}_i}^{\top} R_{B_i}(x) \bm{e}_{\widehat{u}_i}$.
Taking the Schur complement of
$\operatorname{diag}(xI - A(B_1), xI - A(B_2))$ gives
\begin{equation}\label{eq:characteristic-polynomial}
\det (xI - A(G)) = \det (xI - A(B_1)) \det (xI - A(B_2)) \det S_C(x),
\end{equation}
where $S_C(x) = xI - A(C) - g_1(x) \bm{e}_{\widehat{v}_1} \bm{e}_{\widehat{v}_1}^{\top} 
- g_2(x) \bm{e}_{\widehat{v}_2} \bm{e}_{\widehat{v}_2}^{\top}$. Set
$U = [\bm{e}_{\widehat{v}_1}, \bm{e}_{\widehat{v}_2}]$. Then
\[
S_C(x) = xI - A(C) - U\, \mathrm{diag} (g_1(x), g_2(x))\, U^{\top}.
\]
It follows that
\[
\det S_C(x) = \det (xI - A(C))
\det\big(I_2 - \mathrm{diag} (g_1(x), g_2(x))\, U^{\top} R_C(x) U\big).
\]
Since $U^{\top} R_C(x) U = \left[
\begin{smallmatrix} r_{1,C}(x) & \tau_C(x) \\ \tau_C(x) & r_{2,C}(x) \end{smallmatrix}\right]$,
the last $2\times 2$ determinant above is
\begin{align*}
& ~\det \begin{bmatrix}
1 - g_1(x) r_{1,C}(x) & - g_1(x) \tau_C(x) \\
-g_2(x) \tau_C(x) & 1 - g_2(x)r_{2,C}(x)
\end{bmatrix} \\
= & ~ (1 - g_1(x)r_{1,C}(x)) (1 - g_2(x)r_{2,C}(x)) - g_1(x)g_2(x)\tau_C(x)^2.
\end{align*}
Cramer's rule gives $g_i(x) = \phi(B_i - \widehat{u}_i, x)/\phi(B_i, x)$,
and therefore, wherever $g_i(x)\neq 0$,
$Q_{B_i, \widehat{u}_i}(x) = 1/g_i(x)$.
It follows from $F_{i,C}(x) = Q_{B_i, \widehat{u}_i} (x) - r_{i,C}(x)$ that
\begin{align*}
& ~ \big(1 - g_1(x)r_{1,C}(x)\big) \big(1 - g_2(x)r_{2,C}(x)\big) - g_1(x)g_2(x)\tau_C(x)^2 \\
= & ~ g_1(x)g_2(x)\big(F_{1,C}(x) F_{2,C}(x) - \tau_C(x)^2\big).
\end{align*}
Combining with \eqref{eq:characteristic-polynomial}, we have 
\[
\phi(G, x) = \phi(B_1-\widehat{u}_1, x) \phi(B_2-\widehat{u}_2, x) \phi(C, x)
\big(F_{1,C}(x) F_{2,C}(x) - \tau_C(x)^2\big).
\]
Thus, if $x\notin\operatorname{spec}(B_1)\cup\operatorname{spec}(B_2)\cup
\operatorname{spec}(C)$ is an eigenvalue of $G$, then
\[
F_{1,C}(x) F_{2,C}(x) - \tau_C(x)^2 = 0.
\]

Put $G_0 = G-\{\widehat{u}_1, \widehat{u}_2\}$, we have
$\lambda_1(G_0) = \max\{\lambda_1(B_1 - \widehat{u}_1), \lambda_1(C), \lambda_1(B_2 - \widehat{u}_2)\}
< \lambda_2$, where the last inequality follows the argument similar to the proof of Lemma \ref{lem:petal-compression}.
On the other hand, Cauchy interlacing gives
$\lambda_3(G)\leq\lambda_1(G_0) < \lambda_2$.
Thus $G$ has exactly two eigenvalues above $\lambda_1(G_0)$.
For $x > \lambda_1(G_0)$, $\phi(G, x) = 0$ is equivalent to
$F_{1,C}(x) F_{2,C}(x) - \tau_C(x)^2 = 0$. Hence, $\lambda_1(G)$ and $\lambda_2(G)$
are the two roots of $F_{1,C}(x) F_{2,C}(x) - \tau_C(x)^2 = 0$.

Put $Q_i(t):=\frac{\phi(B_i,t)}{\phi(B_i-\widehat u_i,t)}$.
We use $I:=[q-K,q+K]$, where $K$ is a sufficiently
large fixed constant. The estimates below show that all relevant
spectral parameters belong to $I$.

\begin{lemma}\label{lem:loaded-corridor-poles}
Under the standing assumptions, the following assertions hold for all
sufficiently large $q$.
\begin{enumerate}
\item[$(1)$] For $i=1,2$, the function $F_{i,C}$ has a unique finite zero
$\vartheta_i$ in $I$, which is also its unique zero in
$(\lambda_1(B_i), \infty)$. Moreover,
\begin{equation}\label{eq:lcp-spectral-interpretation}
\vartheta_i
=\lambda_1\bigl(G[V(B_i)\cup V(C)]\bigr)
=\lambda_1(B_i) + \Theta (z_i(\widehat{u}_i)^2/q).
\end{equation}
\item[$(2)$] If $D_{i,C}:=F_{i,C}'(\vartheta_i)$, then
\begin{equation}\label{eq:loaded-pole-and-slope}
\vartheta_i
=\lambda_1(B_i) + z_i(\widehat{u}_i)^2 r_{i,C}(\vartheta_i)\bigl(1+O(q^{-2})\bigr),
\qquad
D_{i,C} = z_i(\widehat{u}_i)^{-2}\big(1 + O(q^{-2})\big).
\end{equation}
In fact, if $\widehat{\bm z}_i$ is the positive unit Perron vector
of $G[V(B_i)\cup V(C)]$, then
$D_{i,C}=\widehat z_i(\widehat u_i)^{-2}$.

\item[$(3)$] For $i=1,2$, we have $\lambda_1(B_i) < \lambda_2 < \vartheta_i < \lambda_1$.
For every $x_*\in [\lambda_2, \lambda_1]$, we have
\begin{equation}\label{eq:loaded-two-pole-formula}
\lambda_1 - \lambda_2 = \bigl(1+O(q^{-2})\bigr)
\sqrt{(\vartheta_1-\vartheta_2)^2+
\frac{4\tau_C(x_*)^2}{D_{1,C}D_{2,C}}}.
\end{equation}
\end{enumerate}
\end{lemma}

\begin{proof}
Fix $i\in\{1,2\}$, let $\widehat B_i:=G[V(B_i)\cup V(C)]$, and
put $\rho_i:=\lambda_1(\widehat B_i)$, $\beta_i:=\lambda_1(B_i)$. Since $\widehat B_i$ is
connected and contains $B_i$, we have $\rho_i>\beta_i$.
Also, $\rho_i\leq\beta_i+1$ by Weyl's inequality, since
$\lambda_1(C)\leq4<\beta_i$ and $\widehat B_i$ is obtained from
$B_i\sqcup C$ by adding one edge. 

Write a positive Perron vector of $\widehat{B}_i$ in blocks
$\bm{x}$, $\bm{y}$ indexed by $V(B_i)$, $V(C)$, respectively. 
Its eigenvalue equation is
\[
(\rho_iI - A(B_i))\bm{x} = y_{\widehat{v}_i}\bm{e}_{\widehat{u}_i}, \qquad
(\rho_iI - A(C))\bm{y} = x_{\widehat{u}_i}\bm{e}_{\widehat v_i}.
\]
Solving the second equation gives
$y_{\widehat v_i}=r_{i,C}(\rho_i)x_{\widehat u_i}$.
Substituting this into the first equation, solving for $\bm x$,
and taking the $\widehat u_i$-coordinate, we obtain
$1=r_{i,C}(\rho_i)g_i(\rho_i)$, where
$g_i(t):=\bm e_{\widehat u_i}^{\top}(tI-A(B_i))^{-1}
\bm e_{\widehat u_i}$. Cramer's rule gives $Q_i(t)=1/g_i(t)$ for
$t>\beta_i$, and hence $F_{i,C}(\rho_i)=0$.

For $t>\beta_i$, we have $g_i(t)>0$,
$g_i'(t)=-\|(tI-A(B_i))^{-1}\bm e_{\widehat u_i}\|^2<0$, and
$r_{i,C}'(t)=-\|R_C(t)\bm e_{\widehat v_i}\|^2<0$.
Consequently,
$F_{i,C}'(t)=-g_i'(t)/g_i(t)^2-r_{i,C}'(t)>0$.
Thus, $\rho_i$ is the unique zero of $F_{i,C}$ above $\beta_i$;
we denote it by $\vartheta_i$.

Choose an orthonormal eigenbasis
$\bm{z}_i^{(1)},\ldots,\bm{z}_i^{(m_i)}$ of $A(B_i)$, with
$\bm{z}_i^{(1)} = \bm{z}_i$. The spectral theorem gives
\begin{equation}\label{eq:green-expansion-detailed}
g_i(t)=\frac{a_i^2}{t-\beta_i}+h_i(t), \qquad
h_i(t):=\sum_{k=2}^{m_i}
\frac{z_i^{(k)}(\widehat u_i)^2}{t-\lambda_k(B_i)}, \qquad 
a_i = z_i(\widehat{u}_i).
\end{equation}
Each denominator in $h_i(t)$ is bounded below by a positive
constant times $q$, uniformly for $t\in I$. Hence,
$0\leq h_i(t)=O(q^{-1})$, and $|h_i'(t)| = O(q^{-2})$ $(t\in I)$.
Since $\|A(C)\|\leq 4$, the Neumann series gives
$R_C(t)=t^{-1}\sum_{k\geq 0} A(C)^k/t^k$. The diagonal of $A(C)$
is zero, so
\begin{equation}\label{eq:lcp-diagonal-bounds}
r_{i,C}(t) = t^{-1}+O(q^{-3}), \qquad
|r_{i,C}'(t)|\leq (t-4)^{-2} = O(q^{-2})
\qquad (t\in I).
\end{equation}
Writing $\delta := t - \beta_i$, since $Q_i(t) = 1/g_i(t)$, we obtain
\begin{equation}\label{eq:Q-local-formula}
Q_i(\beta_i+\delta)
=\frac{\delta}{a_i^2+\delta h_i(\beta_i+\delta)}.
\end{equation}
At $\delta=0$, the continuous extension has value $0$. By \eqref{eq:Q-local-formula},
the equation $F_{i,C}(\vartheta_i)=0$ yields
\[
(\vartheta_i-\beta_i)
\bigl(1-r_{i,C}(\vartheta_i)h_i(\vartheta_i)\bigr)
=a_i^2r_{i,C}(\vartheta_i).
\]
Since $r_{i,C}(\vartheta_i)h_i(\vartheta_i)=O(q^{-2})$, it follows
that
\begin{equation}\label{eq:lcp-shift-size}
\vartheta_i-\beta_i
=a_i^2r_{i,C}(\vartheta_i)\bigl(1+O(q^{-2})\bigr)
=\Theta(a_i^2/q).
\end{equation}
There is no further finite zero of $F_{i,C}$ in $I$.
Indeed, a zero $t<\beta_i$ would satisfy
$g_i(t)=1/r_{i,C}(t)=\Theta(q)$, whereas
\eqref{eq:green-expansion-detailed} gives
$g_i(t)<h_i(t)=O(q^{-1})$. At $t=\beta_i$, we have
$F_{i,C}(\beta_i)=-r_{i,C}(\beta_i)<0$. This proves~(1).

(2) Differentiating \eqref{eq:Q-local-formula}, we obtain
\begin{equation}\label{eq:Q-derivative}
Q_i'(\beta_i+\delta)
=\frac{a_i^2-\delta^2h_i'(\beta_i+\delta)}
{\bigl(a_i^2+\delta h_i(\beta_i+\delta)\bigr)^2}.
\end{equation}
For any fixed $L>0$ and $0\leq\delta\leq La_i^2/q$, we have
$\delta h_i/a_i^2=O(q^{-2})$ and
$\delta^2|h_i'|/a_i^2=O(q^{-4})$. Therefore, for $\beta_i\leq t\leq\beta_i + La_i^2/q$,
\begin{equation}\label{eq:lcp-uniform-slope}
Q_i'(t) = a_i^{-2} \big(1 + O(q^{-2})\big), \quad
F_{i,C}'(t) = a_i^{-2}\big(1 + O(q^{-2})\big).
\end{equation}
Taking $t=\vartheta_i$ proves the estimate for $D_{i,C}$.

For the exact identity in (2), we have for $t>\vartheta_i$,
\[
\bm{e}_{\widehat u_i}^{\top}(tI-A(\widehat{B}_i))^{-1}
\bm{e}_{\widehat{u}_i} = \frac{g_i(t)}{1-r_{i,C}(t)g_i(t)}
=\frac1{F_{i,C}(t)}.
\]
As $t\to\vartheta_i$ from above, the spectral expansion of the
left-hand side is
$\widehat z_i(\widehat u_i)^2/(t-\vartheta_i)+O(1)$.
Taking reciprocals and differentiating at $\vartheta_i$ yields
$D_{i,C}=\widehat z_i(\widehat u_i)^{-2}$. This completes (2).

\smallskip
(3) Since $\widehat B_i$ is a proper induced subgraph of the connected
graph $G$, we have $\vartheta_i<\lambda_1$.
By \eqref{eq:lcp-shift-size} and $a_i\geq q^{-8}$, we have
$\vartheta_i-\beta_i\geq c_0q^{-17}$ for some fixed $c_0>0$.
Since $\gamma=o(q^{-20})$, we obtain
$\lambda_2=\lambda_1-\gamma>\vartheta_i-\gamma>\beta_i$.
Fix $i$ and let $j\neq i$. Deleting $\widehat{u}_j$ from $G$
leaves the disjoint union of $\widehat{B}_i$ and
$B_j-\widehat{u}_j$. Cauchy interlacing gives
$\lambda_2\leq\lambda_1(G-\widehat{u}_j)
= \max\{\vartheta_i,\lambda_1 (B_j-\widehat{u}_j)\}$.
Since $\lambda_1(B_j-\widehat u_j)\leq\beta_j<\lambda_2$,
it follows that $\beta_i < \lambda_2\leq\vartheta_i<\lambda_1$, $i=1,2$.

Since $\lambda_1(G)$ and $\lambda_2(G)$
are the two roots of $F_{1,C}(x) F_{2,C}(x) - \tau_C(x)^2 = 0$, and $\tau_C(t)>0$,
we have $F_{i,C}(\lambda_2)\neq0$, and $\beta_i < \lambda_2\leq\vartheta_i<\lambda_1$ is strict at its middle
inequality.

\smallskip
Put $J:=[\lambda_2,\lambda_1]$. For $t\in J$, we have
$|t-\vartheta_i|\leq\gamma$. Together with
\eqref{eq:lcp-shift-size}, this gives
$0<t-\beta_i=O(a_i^2/q)$, since $a_i^2/q\geq q^{-17}$.
Thus, \eqref{eq:lcp-uniform-slope} holds throughout $J$.
The mean value theorem yields
\begin{equation}\label{eq:lcp-linearization}
F_{i,C}(t)=D_{i,C}(t-\vartheta_i)\bigl(1+O(q^{-2})\bigr)
\qquad(t\in J).
\end{equation}

\begin{claim}\label{claim:lcp-log-derivative}
For $t\in I$, we have
\begin{equation}\label{eq:lcp-log-derivative}
0\leq -\frac{\tau_C'(t)}{\tau_C(t)}\leq\mathrm{tr}\, R_C(t)
\leq\frac{s}{t-4}=O(\log q).
\end{equation}
\end{claim}

\begin{proof}[Proof of Claim \ref{claim:lcp-log-derivative}]
Fix $t\in I$ and write $R:=R_C(t)$. The Neumann series shows
that $R$ is entrywise nonnegative. For $u,v,w\in V(C)$, we have
$R_{uw}R_{wv}\leq R_{uv}R_{ww}$. Indeed, for $w\notin\{u,v\}$, we have
\[
R_{uv}-\frac{R_{uw}R_{wv}}{R_{ww}}
=\bigl((tI-A(C-w))^{-1}\bigr)_{uv}\geq0,
\]
where the last inequality again follows from the Neumann series.
Since $R_C'(t)=-R_C(t)^2$, we obtain
\[
0\leq-\tau_C'(t)
=\sum_{w\in V(C)} R_{\widehat{v}_1w} R_{w\widehat{v}_2}
\leq\tau_C(t)\, \mathrm{tr}\, R.
\]
Finally, each eigenvalue of $A(C)$ is at most $4$, and hence
$\mathrm{tr}\, R\leq s/(t-4)$. 
\end{proof}

Fix $x_*\in J$. Integrating \eqref{eq:lcp-log-derivative} gives
\[
\left|\log\frac{\tau_C(t)}{\tau_C(x_*)}\right|
\leq O(\log q)|t-x_*|
\leq O(\gamma\log q)=o(q^{-2})
\qquad(t\in J).
\]
Consequently,
$\tau_C(t)=\tau_C(x_*)\bigl(1+o(q^{-2})\bigr)$ uniformly on $J$.
Combining this with \eqref{eq:lcp-linearization}, at $t=\lambda_1$ and
$t=\lambda_2$ we obtain
\[
(\lambda_j-\vartheta_1)(\lambda_j-\vartheta_2) = \frac{\tau_C(x_*)^2}{D_{1,C}D_{2,C}}
\big(1 + O(q^{-2})\big) \qquad (j=1,2).
\]
It follows that
\[
(\lambda_j - (\vartheta_1+\vartheta_2)/2)^2
= \frac{(\vartheta_1-\vartheta_2)^2}{4}
+ \frac{\tau_C(x_*)^2}{D_{1,C}D_{2,C}}
\big(1+O(q^{-2})\big) \qquad (j=1,2).
\]
Both terms on the right are nonnegative, so taking square roots
introduces only a relative error $O(q^{-2})$ in the whole
expression, even when $\vartheta_1=\vartheta_2$.
Since $\lambda_2 < (\vartheta_1+\vartheta_2)/2 < \lambda_1$, subtracting 
$\lambda_2-(\vartheta_1+\vartheta_2)/2$ from 
$\lambda_1-(\vartheta_1+\vartheta_2)/2$ proves \eqref{eq:loaded-two-pole-formula}.
\end{proof}

\begin{remark}\label{rem:lcp-one-terminal}
The proofs of parts~\textnormal{(1)} and~\textnormal{(2)}, and the
uniform estimate \eqref{eq:lcp-uniform-slope}, use only
the assumptions on $B_i$ and the bound $\Delta(C)\leq4$.
They do not use the small-gap assumption. Accordingly, these
conclusions remain valid when $C$ is replaced by any connected
graph of maximum degree at most $4$, attached to $B_i$ by one
edge at a designated vertex. We shall use this observation for
paths. The small-gap assumption is used in part~\textnormal{(3)}.
\end{remark}

\begin{proposition}\label{prop:corridor-is-path}
The induced subgraph $C=G[V(G)\setminus(V(B_1)\cup V(B_2))]$ is a path with
endvertices $\widehat v_1$ and $\widehat v_2$.
\end{proposition}

\begin{proof}
Suppose to the contrary that $C$ is not a path with these
endvertices. Put $x_*:=(\lambda_1+\lambda_2)/2$.
By Lemma~\ref{lem:loaded-corridor-poles}, we have $x_*=q+O(1)$ and
\begin{equation}\label{eq:cp-original-root-location}
|\vartheta_i-x_*|\leq\gamma/2
\qquad(i=1,2).
\end{equation}
We first obtain a lower bound for $\gap(G)$. We shall then
construct two symmetric graphs of order $n$ and show that at
least one has a smaller spectral gap.

\begin{claim}\label{claim:cp-original-gap}
For all sufficiently large $q$, we have
\begin{equation}\label{eq:cp-original-gap}
\gap(G)\geq\frac{q\,a_1a_2}{p_s(x_*)}.
\end{equation}
\end{claim}

\begin{proof}[Proof of Claim~\ref{claim:cp-original-gap}]
By \eqref{eq:loaded-pole-and-slope} and
\eqref{eq:loaded-two-pole-formula}, we obtain
\begin{equation}\label{eq:cp-coupling-lower}
\gap(G)\geq2a_1a_2\tau_C(x_*)\bigl(1-O(q^{-2})\bigr).
\end{equation}
If $\dist_C(\widehat v_1,\widehat v_2)=s-1$, a shortest path
between the attachment vertices would contain every vertex of
$C$. Any additional edge would shorten their distance, so $C$
would be precisely this path, contrary to our assumption.
Thus, $d:=\dist_C(\widehat v_1,\widehat v_2)\leq s-2$.
The term corresponding to one shortest path in
$R_C(x_*)=x_*^{-1}\sum_{k\geq0}A(C)^k/x_*^k$ gives
$\tau_C(x_*)\geq x_*^{-d-1}\geq x_*^{-(s-1)}$.
On the other hand, we have
$p_s(x_*)\geq\sigma(x_*)^s$. Since
$\sigma(x_*)/x_*=1+O(q^{-2})$ and $s=O(q\log q)$, we obtain
\[
\tau_C(x_*)p_s(x_*)
\geq x_*\Big(\frac{\sigma(x_*)}{x_*}\Big)^s
= q\big(1 + o(1)\big).
\]
Combining this with \eqref{eq:cp-coupling-lower} proves
\eqref{eq:cp-original-gap} for all sufficiently large $q$.
\end{proof}

For $i\in\{1,2\}$, put $s_i:=n-2m_i$. Then $s_i=s+O(1)$,
$s_1+s_2=2s$, and $s_i\geq2$ for sufficiently large $q$.
Let $G^{(i)}$ be obtained from two vertex-disjoint copies of the
rooted graph $(B_i,\widehat u_i)$ and a disjoint path $P_{s_i}$
by joining the two copies of $\widehat u_i$ to the two
endvertices of the path, respectively. The graph $G^{(i)}$ is
connected and has $2m_i+s_i=n$ vertices. 

For $k\geq2$ and $t>2$, put $r_k(t):=p_{k-1}(t)/p_k(t)$ and
$\tau_k(t):=1/p_k(t)$. For each $i$, define
$\widehat F_i(t):=Q_i(t)-r_{s_i}(t)$.

\begin{claim}\label{claim:cp-root-localization}
For each $i$, the function $\widehat F_i$ has a unique zero
$\eta_i>\beta_i$. Moreover,
\begin{equation}\label{eq:cp-root-localization}
|\eta_i-\vartheta_i|=O(a_i^2q^{-3}),\qquad
|\eta_i-x_*|=O(q^{-3}).
\end{equation}
\end{claim}

\begin{proof}[Proof of Claim~\ref{claim:cp-root-localization}]
Apply parts~\textnormal{(1)} and~\textnormal{(2)} of
Lemma~\ref{lem:loaded-corridor-poles} to one copy of $B_i$
attached to $P_{s_i}$ at an endvertex, as justified in
Remark~\ref{rem:lcp-one-terminal}. We obtain a unique zero
$\eta_i>\beta_i$ of $\widehat F_i$, with
$\eta_i-\beta_i=\Theta(a_i^2/q)$.
Let $L$ be a sufficiently large fixed constant and put
$J_i:=[\beta_i,\beta_i+La_i^2/q]$.
Both $\eta_i$ and $\vartheta_i$ belong to $J_i$.
The uniform estimate \eqref{eq:lcp-uniform-slope}
for the path connector gives
\begin{equation}\label{eq:cp-uniform-derivative}
\widehat F_i'(t)=a_i^{-2}\bigl(1+O(q^{-2})\bigr)
\qquad(t\in J_i).
\end{equation}
By \eqref{eq:lcp-diagonal-bounds} and the same Neumann-series
estimate for a path, both $r_{i,C}(t)$ and $r_{s_i}(t)$ equal
$t^{-1}+O(q^{-3})$ uniformly on $I$. Hence,
$\widehat{F}_i(\vartheta_i)
=r_{i,C}(\vartheta_i)-r_{s_i}(\vartheta_i)=O(q^{-3})$.
Since $\widehat F_i(\eta_i)=0$, the mean value theorem and
\eqref{eq:cp-uniform-derivative} give
$|\eta_i-\vartheta_i|=O(a_i^2q^{-3})$.
Together with \eqref{eq:cp-original-root-location}, this yields
$|\eta_i-x_*|\leq O(a_i^2q^{-3})+\gamma/2=O(q^{-3})$.
\end{proof}

\begin{claim}\label{claim:cp-symmetric-gaps}
For $i=1,2$, we have
\begin{equation}\label{eq:cp-symmetric-gap}
\gap(G^{(i)})
=\frac{2a_i^2}{p_{s_i}(x_*)}\bigl(1+O(q^{-2})\bigr).
\end{equation}
\end{claim}

\begin{proof}[Proof of Claim~\ref{claim:cp-symmetric-gaps}]
Fix $i$. Consider the graph $G^{(i)}$, we have 
that a number $t>\beta_i$ is an eigenvalue of $G^{(i)}$ if and
only if
\[
\det\begin{bmatrix}
\widehat{F}_i(t) & -\tau_{s_i}(t) \\
-\tau_{s_i}(t) & \widehat{F}_i(t)
\end{bmatrix}=0.
\]
Equivalently, $t$ satisfies one of the two equations
\begin{equation}\label{eq:cp-symmetric-equations}
Q_i(t)=r_{s_i}(t)+\tau_{s_i}(t),\qquad
Q_i(t)=r_{s_i}(t)-\tau_{s_i}(t).
\end{equation}
For either choice of sign, Lemma \ref{lem:R(F-x)} gives
\[
r_{s_i}(t)\pm\tau_{s_i}(t)
=\frac12(\bm e_1\pm\bm e_{s_i})^{\top}
R_{P_{s_i}}(t)(\bm e_1\pm\bm e_{s_i}).
\]
These functions are positive and strictly decreasing for $t>2$,
since $R_{P_{s_i}}(t)$ is positive definite and its derivative
is $-R_{P_{s_i}}(t)^2$. Also,
$0<r_{s_i}(t)\pm\tau_{s_i}(t)\leq(t-2)^{-1}$.
In contrast, $Q_i$ is strictly increasing on $(\beta_i,\infty)$,
with $Q_i(\beta_i)=0$ and $Q_i(t)\to\infty$ as $t\to\infty$.
Thus, each equation in \eqref{eq:cp-symmetric-equations} has a
unique root above $\beta_i$. Denote these roots by $\mu_{i,+}$
and $\mu_{i,-}$, respectively.

By \eqref{eq:Q-local-formula},
$Q_i(\beta_i+La_i^2/q)=(L/q)(1+O(q^{-2}))$.
The bound $(t-2)^{-1}=q^{-1}(1+O(q^{-1}))$ on $J_i$ shows that
both roots belong to $J_i$ when $L$ is sufficiently large.
Moreover, $\widehat F_i(\mu_{i,\pm})=\pm\tau_{s_i}(\mu_{i,\pm})$
and $\widehat F_i$ is strictly increasing, so
$\mu_{i,-}<\eta_i<\mu_{i,+}$.
Deleting the two terminal attachment vertices from $G^{(i)}$
leaves $2(B_i-\widehat u_i)\sqcup P_{s_i}$, whose spectral
radius is less than $\beta_i$.
By interlacing, $G^{(i)}$ has at most two eigenvalues above
$\beta_i$. Since both roots reconstruct eigenvectors, we obtain
$\mu_{i,+}=\lambda_1(G^{(i)})$ and
$\mu_{i,-}=\lambda_2(G^{(i)})$.

The mean value theorem and \eqref{eq:cp-uniform-derivative} give
\begin{equation}\label{eq:cp-root-displacements}
|\mu_{i,\pm}-\eta_i|
=a_i^2\tau_{s_i}(\mu_{i,\pm})\bigl(1+O(q^{-2})\bigr).
\end{equation}
Since $\tau_{s_i}(t)\leq\sigma(t)^{-s_i}$, $s_i=\Theta(q\log q)$ and $\sigma(t)=\Theta(q)$ uniformly
on $I$, it follows that
$\tau_{s_i}(t)=o(q^{-M})$ uniformly on $I$ for every fixed
$M>0$. Thus, \eqref{eq:cp-root-displacements} and
Claim~\ref{claim:cp-root-localization} imply that
$|\mu_{i,\pm}-x_*|=O(q^{-3})$.

To evaluate the path terms at the common parameter $x_*$,
factor $p_{s_i}$ over the adjacency eigenvalues of $P_{s_i}$.
For $t\in I$, we obtain
\begin{equation}\label{eq:cp-path-log-derivative}
0<\frac{p_{s_i}'(t)}{p_{s_i}(t)}
=\sum_{k=1}^{s_i}\frac1{t-\lambda_k(P_{s_i})}
\leq\frac{s_i}{t-2}=O(\log q).
\end{equation}
Integrating between $\mu_{i,\pm}$ and $x_*$, we infer that
\[
\tau_{s_i}(\mu_{i,\pm})
=\tau_{s_i}(x_*)\bigl(1+O(q^{-3}\log q)\bigr)
=\frac{1+o(q^{-2})}{p_{s_i}(x_*)}.
\]
Substituting this into \eqref{eq:cp-root-displacements} and
adding the two displacements gives
\eqref{eq:cp-symmetric-gap}.
\end{proof}

We now compare the two graphs $G^{(1)}$ and $G^{(2)}$ simultaneously.
Since $s_1+s_2=2s$ and $s_i=s+O(1)$, we have
\begin{equation}\label{eq:cp-path-product}
\frac{p_{s_1}(x_*)p_{s_2}(x_*)}{p_s(x_*)^2}
=
\frac{(1-\sigma(x_*)^{-2s_1-2})(1-\sigma(x_*)^{-2s_2-2})}
{(1-\sigma(x_*)^{-2s-2})^2}
=1+o(q^{-2}).
\end{equation}
Indeed, the powers of $\sigma(x_*)$ and the common denominators
cancel, and each of the remaining negative powers is smaller
than every fixed negative power of $q$.
By Claim~\ref{claim:cp-symmetric-gaps} and
\eqref{eq:cp-path-product}, we obtain
\begin{align*}
\min_{i\in\{1,2\}}\gap(G^{(i)})
&\leq\sqrt{\gap(G^{(1)})\gap(G^{(2)})}\\
&=\frac{2a_1a_2}{\sqrt{p_{s_1}(x_*)p_{s_2}(x_*)}}
\big(1+O(q^{-2})\big) \\
&=\frac{2a_1a_2}{p_s(x_*)}\big(1+O(q^{-2})\big) < \frac{q\,a_1a_2}{p_s(x_*)}
\leq\gap(G)
\end{align*}
for all sufficiently large $q$, where the last inequality is
Claim~\ref{claim:cp-original-gap}.
Both comparison graphs are connected and have order $n$.
This contradicts the minimality of $\gap(G)$.
Thus, $C$ is a path with endvertices $\widehat v_1$ and
$\widehat v_2$.
\end{proof}

\subsection{The induced subgraph $G[V(B_i)]$ is a clique}

Let $F$ be a connected graph with $u\in V(F)$. For $s\geq 1$, let
$X_s(F,u)$ be the graph obtained from $F$ by adding at $u$ a pendant path
$u=y_0,y_1,\ldots,y_s$ of $s$ edges. Let
$\lambda:=\lambda_1(X_s(F,u))$
and let $\bm{z}$ be the unit Perron vector of $X_s(F,u)$.
Write $\eta:= z_{y_s}$ for the Perron-vector coordinate at the leaf vertex of the
pendant path.

We assume that $|V(F)| = m$ and that there are constants $C,D>0$,
independent of $m$, such that
$m - C\leq\lambda\leq m$, $s = \Theta(m\log m)$.
Choose $w\in V(F)$ such that $z_w = \max_{v\in V(F)} z_v$.
We further assume that $\mathrm{dist}_F(u,w)\leq D$.

For comparison, let $\mu_m:=\lambda_1(X_s(K_m,u))$,
where $u$ is any vertex of $K_m$. Let $\bm{z}^{(m)}$ be the
unit Perron eigenvector of $X_s(K_m,u)$, and define
$a_m = z^{(m)}_u$. Thus $a_m$ is the Perron-vector coordinate at the vertex of $K_m$
to which the pendant path is attached.

We will now show that each $B_i$ induces a clique. To outline the proof: we may understand the eigenvalue and eigenvector information of $X_s(F, u)$ or $X_s(K_m,u)$ very well; this allows us to compute the spectral gap of the graph obtained by joining two copies of either $G_1$ or $G_2$ by a path; we may then compare this to $\gap(G)$ and subsequently compare it to the double kite. 

\begin{lemma}\label{lem:complete-block-estimates}
Let $\bm{z}$ be the positive unit Perron vector of $X_s(K_m,u)$. Then
\begin{equation}
\lambda_1(X_s(K_m, u)) = m - 1 + O(m^{-2}), \qquad
a_m=z_u^{(m)} = m^{-1/2}\big(1 + O(m^{-2})\big).
\label{eq:complete-block-pole-mass}
\end{equation}
Moreover, if $e$ is any edge of $K_m$, then
\begin{equation}
\lambda_1(X_s(K_m,u)) - \lambda_1(X_s(K_m-e,u))\geq\frac{\kappa}{m}
\label{eq:missing-edge-pole-gain}
\end{equation}
for an absolute $\kappa>0$ and all sufficiently large $m$.
\end{lemma}

\begin{proof}
Write $z_i:= z_{y_i}^{(m)}$, $i=0,1,\ldots,s$.
The eigenvalue\,--\,eigenvector equation at $y_s$ is $\mu_m z_s = z_{s-1}$. Thus
$z_{s-1} = p_1(\mu) z_s$. At the vertex $y_{s-1}$,
$\mu_m z_{s-1} = z_{s-2} + z_s$, so
$z_{s-2} =(\mu_m^2 - 1) z_s = p_2(\mu_m) z_s$.
By induction, we have $z_{s-k} = p_k(\mu_m) z_s$
($0\leq k\leq s$). Taking $k=s$ and $k=s-1$, respectively, gives
$z_u = p_s(\mu) z_s$, $z_1 = p_{s-1}(\mu) z_s$. Therefore
\begin{equation}\label{eq:first-end-ratio}
z_1 = \frac{p_{s-1}(\mu_m)}{p_s(\mu_m)} z_u = r_s(\mu_m) z_u, \quad 
\text{where}~~ r_s(x):= \frac{p_{s-1}(x)}{p_s(x)}.
\end{equation}

The recurrence for $p_k(x)$ gives, for $k\geq2$,
\begin{equation}\label{eq:complete-ratio-recurrence}
r_k(x) = \frac{p_{k-1}(x)}{xp_{k-1}(x)-p_{k-2}(x)}
= \frac{1}{x-r_{k-1}(x)}.
\end{equation}
It follows that $0 < r_k(x)\leq\frac1{x-1}$ for $x\geq 3$. Furthermore,
\[
r_k(x) - \frac{1}{x} = \frac{1}{x-r_{k-1}(x)} - \frac{1}{x}
= \frac{r_{k-1}(x)}{x (x - r_{k-1}(x))}.
\]
Since $r_{k-1}(x)=O(x^{-1})$, uniformly in $k$, we obtain
$r_k(x) - 1/x = O(x^{-3})$. 

Since all $m-1$ clique vertices different from $u$ have the same Perron coordinate of $\bm{z}$, 
denote this common coordinate by $b$.
Let $v\in V(K_m)\setminus\{u\}$. The eigenvalue\,--\,eigenvector equation at $v$ is
$\mu_m b = z_u + (m-2)b$. Hence $z_u = (\mu_m - m + 2)b$.
At the vertex $u$, we have $\mu_m z_u = (m-1)b + z_1$. Using $z_1 = r_s(\mu_m)z_u$, 
we obtain $(\mu_m - r_s(\mu_m)) z_u = (m - 1)b$. Combining with $z_u = (\mu_m - m + 2)b$ gives
\begin{equation}\label{eq:temporary-2}
\big(\mu_m - r_s(\mu_m)\big) (\mu_m - m + 2) = m - 1.
\end{equation}

Now, we estimate $\mu_m$. Since $K_m$ is a proper subgraph of $X_s(K_m, u)$,
$\mu_m > m - 1$. Also, the maximum degree of $X_s(K_m, u)$ is $m$, attained at $u$. Hence
$\mu_m\leq m$.
Write $\mu_m = m - 1 + \varepsilon$, where $0 < \varepsilon\leq 1$.
Equation \eqref{eq:temporary-2} becomes
$(m - 1 + \varepsilon - r_s(\mu_m)) (1 + \varepsilon) = m - 1$.
Expanding and cancelling $m-1$ gives
$m\varepsilon + \varepsilon^2 = r_s(\mu_m)(1 + \varepsilon)$.
Since $\mu_m = \Theta(m)$, $r_s(\mu_m) = O(m^{-1})$. Hence,
$m\varepsilon=O(m^{-1})$, and hence $\varepsilon = O(m^{-2})$. Therefore $\mu_m = m - 1 + O(m^{-2})$.

For $1\leq j\leq s$, by $z_{s-k} = p_k(\mu_m) z_s$ ($0\leq k\leq s$) we have
\[
\frac{z_j}{z_{j-1}} = \frac{p_{s-j}(\mu_m)}{p_{s-j+1}(\mu_m)}
= r_{s-j+1}(\mu_m).
\]
Since $r_k(\mu_m)\leq (\mu_m - 1)^{-1}$, we have
$z_j\leq z_u/(\mu_m - 1)^j$. Consequently,
\begin{equation}\label{eq:sum-mass}
\sum_{j=1}^s z_j^2 \leq z_u^2 \sum_{j=1}^{\infty}
\frac{1}{(\mu_m - 1)^{2j}} = \frac{z_u^2}{(\mu_m - 1)^2 - 1}
= O(m^{-2}) z_u^2.
\end{equation}

In view of $z_u = (\mu_m - m + 2) b$ and $\mu_m = m - 1 + O(m^{-2})$, we have 
$b = z_u (1 + O(m^{-2}))$.
The normalization $\|\bm{z}\|_2 = 1$, together with \eqref{eq:sum-mass}, gives
$1 = z_u^2 + (m-1)b^2 + O(m^{-2}) z_u^2$. Using $b = z_u (1 + O(m^{-2}))$,
\[
1 = z_u^2 + (m-1)z_u^2\big(1 + O(m^{-2})\big) + O(m^{-2})z_u^2
= z_u^2\big(m + O(m^{-1})\big).
\]
Therefore $z_u = m^{-1/2} (1 + O(m^{-2}))$. 

Next, we prove \eqref{eq:missing-edge-pole-gain}.
Let $e=ab$ be an edge of $K_m$, and let $\bm{w}$ be the positive unit Perron vector of
$X_s(K_m - e, u)$. Write $\lambda:= \lambda_1(X_s(K_m-e, u))$ and $w_j:=w_{y_j}$ for $j=0,1,2,\ldots,s$.
Since $X_s(K_m-e,u)\subseteq X_s(K_m,u)$, we have $\lambda\leq m - 1 + O(m^{-2})$.
The unit vector that is constant on the $m$ clique vertices and zero on
the pendant path gives
\[
 \lambda\geq \frac{2|E(K_m-e)|}{m}
 =m-1-\frac2m.
\]
Consequently $\lambda=m-1+O(m^{-1})$.
Moreover, the same path recurrence as above gives
$\sum_{j=1}^s w_j^2 = O(m^{-2}) w_u^2 = O(m^{-2})$.

\emph{Case 1}: The deleted edge does not contain $u$, i.e., $u\notin e$.
The two vertices $a,b$ have the same Perron coordinate; denote it by $x$. The $m-3$ vertices in
$V(K_m)\setminus\{u,a,b\}$ also have a common coordinate, denoted by $y$. Put
$r:= w_u$. The eigenvalue\,--\,eigenvector equation gives
$\lambda x = (m-3)y + r$, $\lambda y = 2x + (m-4)y + r$, and
$\lambda r = 2x + (m-3)y + w_1$. It follows that $x,y,r$ are asymptotically equal. 
More specifically, we have 
\begin{equation}\label{eq:asy-equal}
\frac{x}{y} = \frac{\lambda + 1}{\lambda + 2}
= 1 + O(m^{-1}), \qquad 
\frac{r}{y} = 1 + O(m^{-1}),
\end{equation}
since $\lambda = m - 1 + O(m^{-1})$.
The normalization of $\bm{w}$, together with
$\sum_{j=1}^s w_j^2 = O(m^{-2})$, now gives
$1 = 2x^2 + (m-3)y^2 + r^2 + O(m^{-2}) = my^2\big(1 + O(m^{-1})\big)$.
Therefore $y = m^{-1/2} (1 + O(m^{-1}))$.
Equation \eqref{eq:asy-equal} gives
$x = m^{-1/2}(1 + O(m^{-1}))$, $r = m^{-1/2}(1 + O(m^{-1}))$. 
In particular, both vertices $a,b$ of the deleted edge have coordinate
$m^{-1/2} (1 + O(m^{-1}))$.
\smallskip

\noindent\emph{Case 2}: The deleted edge $e$ contains $u$.
Suppose $e=ua$. Write $x:=w_a$, $r:=w_u$.
The remaining $m-2$ vertices in $V(K_m)\setminus\{u,a\}$ have a common coordinate $y$.
By the eigenvalue\,--\,eigenvector equations, we have 
$\lambda x = (m-2)y$, $\lambda y = r + x + (m-3)y$, and 
$\lambda r = (m-2)y + w_1$. Moreover,
\[
\frac{x}{y} = \frac{m-2}{\lambda} = 1 + O(m^{-1}), \quad
\frac{r}{y} = \lambda - m + 3 - \frac{x}{y} = 1 + O(m^{-1}).
\]
Using $\|\bm{w}\|=1$ and $\sum_{j=1}^s w_j^2 = O(m^{-2})$, we deduce
$1 = my^2 (1 + O(m^{-1}))$. Therefore
$x = m^{-1/2} (1 + O(m^{-1}))$, $r = m^{-1/2} (1 + O(m^{-1}))$. 
\smallskip

Thus, in either case, if $e=ab$, then
$w_a = m^{-1/2} (1 + o(1))$, and $w_b = m^{-1/2} (1 + o(1))$.
By the Rayleigh principle, $\lambda_1(X_s(K_m, u)) - \lambda_1(X_s(K_m-e, u))
\geq 2w_aw_b = \Omega (1/m)$.
\end{proof}

Let $\bm{z}$ be the positive Perron vector of $X_s(F, u)$. Put
\begin{equation}\label{eq:defect-normalization}
S := \sum_{v\in F} z_v, \qquad
c_v := \frac{(\lambda+1)z_v}{S}, \qquad
\Delta := m-1-\lambda.
\end{equation}

\begin{lemma}\label{lem:exact-defect-identity}
Let $F$ be a connected graph, and $\bm{z}$ be the positive Perron vector of $X_s(F, u)$,
let $w$ be a vertex in $V(F)$ at which $\bm{z}$ is maximum.
If $z_u < z_w$, then
\begin{equation}\label{eq:exact-defect-identity}
0 < c_v\leq 1 \quad (v\in F), \qquad
\Delta = \sum_{v\in F} (1 - c_v).
\end{equation}
Let $u = v_0,v_1,\ldots,v_d=w$ be a shortest path to $w$, and put $c=c_u$. Then
$c_{v_i}\leq\min\{1,\lambda^ic\}$. If $\deg_F(u)\geq 2$, then
\begin{equation}\label{eq:Psi-defect-lower}
\Delta\geq\Psi_d(c):=
d+2-c-\min\{2,\lambda c\} - \sum_{i=2}^d\min\{1,\lambda^ic\}.
\end{equation}
Moreover, $c\geq (1 - O(m^{-1})) \lambda^{-d}$.
\end{lemma}

\begin{proof}
We first prove that $c_v\leq1$.  Let $v\in V(F)\setminus\{u\}$.
The eigenvalue\,--\,eigenvector equation at $v$ gives 
$(\lambda + 1) z_v = \sum_{x\in N_F(v)\cup\{v\}} z_x$.
We may rewrite this as
\begin{equation}\label{eq:eigen-equation-for-z}
(\lambda + 1) z_v = S - \sum_{x\in V(F)\setminus(N_F(v)\cup\{v\})} z_x
\leq S.
\end{equation}
Consequently, $c_v\leq 1$ ($v\in V(F)\setminus\{u\})$. By the definition,
$c_u < c_w$. Since $w\neq u$, inequality \eqref{eq:eigen-equation-for-z} applies to $w$, 
and therefore $c_u < c_w\leq 1$. Hence, $0 < c_v\leq 1$ for all $v\in V(F)$. 

Directly from the definition of $c_v$,
\[
\sum_{v\in V(F)} c_v = \sum_{v\in V(F)} \frac{(\lambda + 1)z_v}{S} = \frac{\lambda+1}{S}
\sum_{v\in V(F)} z_v = \lambda + 1.
\]
Since $|V(F)|=m$, it follows that
$\sum_{v\in V(F)}(1-c_v) = m - (\lambda + 1) = \Delta$. This finishes the proof of \eqref{eq:exact-defect-identity}.

We next compare consecutive coordinates on the path.
For $1\leq i<d$, the eigenvalue\,--\,eigenvector equation at $v_i$ gives
$\lambda z_{v_i}\geq z_{v_{i+1}}$.
For $i=0$, we have
$\lambda z_u=z_{y_1}+\sum_{x\in N_F(u)}z_x\geq z_{v_1}$.
Thus, for every $0\leq i<d$,
$z_{v_{i+1}}\leq\lambda z_{v_i}$.  It follows that
$z_{v_i}\leq\lambda^i z_u$ ($0\leq i\leq d$).
Multiplying by the positive factor $(\lambda+1)/S$, we obtain
$c_{v_i}\leq\lambda^i c$.
On the other hand, $c_{v_i}\leq 1$. Combining the two bounds yields
$c_{v_i}\leq\min\{1, \lambda^ic\}$ ($0\leq i\leq d$).

Now assume that $\deg_F(u)\geq 2$.
Since $v_1$ is one neighbor of $u$, we may choose another neighbor
$a\in N_F(u)\setminus\{v_1\}$. Then $u,a,v_1,\ldots,v_d$
are all distinct. The eigenvalue\,--\,eigenvector equation at $u$ gives
$z_a + z_1\leq\lambda z_u$. Multiplying by $(\lambda + 1)/S$ gives
$c_a + c_{v_1}\leq\lambda c$. Together with $c_v\leq 1$ ($v\in F$), we have $c_a+c_{v_1}\leq 2$.
Consequently, $c_a + c_{v_1}\leq\min\{2, \lambda c\}$. 
By $\Delta = \sum_{v\in F} (1 - c_v)$, we have
\begin{align*}
\Delta
& \geq (1 - c_u) + (1 - c_a) + \sum_{i=1}^{d} (1 - c_{v_i}) \\
& = d + 2 - c - (c_a + c_{v_1}) - \sum_{i=2}^{d} c_{v_i} \\
& \geq d + 2 - c - \min\{2, \lambda c\} - \sum_{i=2}^{d} \min\{1,\lambda^ic\}.
\end{align*}

It remains to prove the lower bound on $c=c_u$. Since $w$ is a
vertex at which $\bm{z}$ is maximal on $F$, we have
$z_w\geq\frac{1}{m} \sum_{v\in V(F)} z_v = S/m$.
Therefore
\[
c_w = \frac{(\lambda+1)z_w}{S} \geq
\frac{(\lambda+1)(S/m)}{S} = \frac{\lambda+1}{m}. 
\]
The hypothesis $m-C\leq\lambda\leq m$, with $C$ fixed, implies
\[
\frac{\lambda + 1}{m} \geq \frac{m-C+1}{m}
= 1 - \frac{C-1}{m} = 1 - O(m^{-1}).
\]
Hence $c_w\geq 1 - O(m^{-1})$. 

Finally, the previously proved bound $c_{v_i}\leq\lambda^ic$, applied
with $i=d$, gives $c_w=c_{v_d}\leq\lambda^dc$.  Combining this with
$c_w\geq1-O(m^{-1})$, we obtain $\lambda^dc\geq1-O(m^{-1})$,
and therefore $c\geq (1-O(m^{-1}))\lambda^{-d}$.
This completes the proof.
\end{proof}

The free-end coordinate has the exact expression
\begin{equation}
 \eta=\frac{z_u}{p_s(\lambda)}.
 \label{eq:free-end-exact}
\end{equation}
Indeed, this follows by solving the three-term recurrence backwards from
$\lambda z_s = z_{s-1}$.

We need two comparisons. Let $\eta_m$ denote the Perron entry of the leaf $y_s$ for
$X_s(K_m, u)$. Let $\eta_{m-1}$ denote the Perron entry of the leaf $y_{s+1}$ for
$X_{s+1}(K_{m-1}, u)$; the latter graph has the same order as $X_s(B,u)$.

\begin{lemma}\label{lem:two-complete-competitors}
If $u$ is not a Perron maximum, then
\begin{equation}
 \log\frac{\eta}{\eta_m}
 \geq \log c+\frac{s}{m+2}\Delta - O(m^{-1}).
 \label{eq:first-complete-comparison}
\end{equation}
If $\Delta = O(1)$ and $A:=s/m$, then, uniformly in $c$ and $\Delta$,
\begin{align*}
\log\frac{\eta}{\eta_m}
& = \log c + A\Delta + o(1), \\
\log\frac{\eta}{\eta_{m-1}}
& = \log c + A\Delta + \log m - A + o(1).
\end{align*}
\end{lemma}

\begin{proof}
For $1\leq j\leq s$, solving the eigenvector recurrence backward
from the leaf $y_s$ gives
$z_j = \frac{p_{s-j}(\lambda)}{p_s(\lambda)} z_u$.
Using \eqref{eq:path-hyperbolic}, we obtain
\[
\frac{p_{s-j}(\lambda)}{p_s(\lambda)}
= \sigma(\lambda)^{-j} \frac{1 - \sigma(\lambda)^{-2(s-j+1)}}{1 - \sigma(\lambda)^{-2(s+1)}}.
\]
Since $\lambda = m + O(1)$, we see $\sigma(\lambda) = \Theta(m)$. Moreover,
$\sigma(\lambda) = \lambda - \lambda^{-1} + O(\lambda^{-3})$.
Since $s = \Theta(m\log m)$,
\[
\Big(\frac{\lambda}{\sigma(\lambda)}\Big)^s
= \exp\Big(O\left(\frac{s}{m^2}\right)\Big) = 1 + o(1).
\]
It follows that, for $1\leq j\leq s$, $z_j\leq 3\lambda^{-j} z_u$. Consequently,
\[
\sum_{j=1}^s z_j^2\leq 9z_u^2\sum_{j=1}^{\infty} \lambda^{-2j}
= \frac{9z_u^2}{\lambda^2 - 1} = O(m^{-2}) z_u^2 \leq O(m^{-2}).
\]

The contribution of the pendant-path edges to the Rayleigh quotient is
$2z_uz_1 + 2\sum_{j=1}^{s-1} z_jz_{j+1}$.
By $z_j\leq 3\lambda^{-j} z_u$, we have $2z_uz_1
\leq 6\lambda^{-1}z_u^2 = O(m^{-1})$, and
\[
2\sum_{j=1}^{s-1} z_jz_{j+1}\leq
18 z_u^2\sum_{j=1}^{\infty} \lambda^{-2j-1} = O(m^{-3})z_u^2.
\]
Therefore the total contribution of the pendant-path edges to
$\bm{z}^{\top} A(X_s(F,u)) \bm{z}$ is $O(m^{-1})$.

We now estimate $S$.  The contribution from the edges inside $F$ is
\[
2\sum_{vx\in E(F)} z_vz_x
\leq 2\sum_{vx\subseteq V(F)} z_vz_x
= \Bigg(\sum_{v\in V(F)} z_v\Bigg)^2 - \sum_{v\in V(F)} z_v^2
= S^2 - \sum_{v\in V(F)} z_v^2.
\]
Combining the contribution from $F$ with the contribution from the
pendant path gives
\[
\lambda \leq S^2 - \sum_{v\in V(F)} z_v^2 + O(m^{-1})
\leq S^2 + O(m^{-1}).
\]
Since $\lambda\geq m-C$, $S^2\geq m - C - O(m^{-1}) = m-O(1)$.
Conversely, by the Cauchy--Schwarz inequality and $\|\bm{z}\| = 1$,
$S^2 \leq m\sum_{v\in V(F)} z_v^2 \leq m$. It follows that $S^2 = m + O_C(1)$.
Taking square roots gives $S = \sqrt{m} (1 + O(m^{-1}))$.

From the definition of $c$, $z_u = \frac{cS}{\lambda + 1}$.
Since $\lambda + 1 = m - \Delta$
and $\lambda\geq m-C$, we have $\Delta = O(1)$.
Lemma~\ref{lem:complete-block-estimates} gives
$z_u^{(m)} = m^{-1/2} (1 + O(m^{-2}))$. Consequently,
\[
\frac{z_u}{z_u^{(m)}} = \frac{cS}{(\lambda + 1) z_u^{(m)}}
= c\,\frac{\sqrt{m} (1 + O_C(m^{-1}))}{(m - \Delta)m^{-1/2}(1 + O(m^{-2}))}
= c\left(1 + O_C(m^{-1})\right).
\]
Taking logarithm gives
\begin{equation}
\log\frac{z_u}{z_u^{(m)}} = \log c+O(m^{-1}).
 \label{eq:detailed-root-log-comparison}
\end{equation}

We next compare the path polynomials. Since
$\eta = z_u/p_s(\lambda)$, and $\eta_m = z_u^{(m)}/p_s(\mu_m)$, we deduce
\begin{equation}\label{eq:detailed-eta-ratio}
\log\frac{\eta}{\eta_m} = \log\frac{z_u}{z_u^{(m)}}
+ \log p_s(\mu_m) - \log p_s(\lambda).
\end{equation}
The polynomial $p_s(x)$ has degree $s$, and all its zeros lie in
$[-2,2]$. If these zeros are denoted by
$\xi_1,\ldots,\xi_s$, then, for $x>2$, $p_s(x) = \prod_{j=1}^s (x - \xi_j)$.
Taking the logarithmic derivative gives
$\frac{p_s'(x)}{p_s(x)} = \sum_{j=1}^s \frac{1}{x - \xi_j}$.
Since $\xi_j\geq - 2$, $x - \xi_j\leq x+2$, we have
\begin{equation}
\frac{\mathrm{d}}{\mathrm{d} x}\log p_s(x) \geq\frac{s}{x+2}.
\label{eq:detailed-log-derivative}
\end{equation}
Since $\lambda = m - 1 - \Delta$ and $\mu_m\geq m-1$,
we have $\mu_m - \lambda\geq\Delta\geq 0$.
Furthermore, the maximum degree of $X_s(K_m,u)$ is $m$, so
$\mu_m\leq m$. It follows from \eqref{eq:detailed-log-derivative} that
\begin{align}
\log p_s(\mu_m) - \log p_s(\lambda)
& = \int_{\lambda}^{\mu_m} \frac{\mathrm{d}}{\mathrm{d} x}\log p_s(x)\,\mathrm{d} x
\geq \int_{\lambda}^{\mu_m} \frac{s}{x+2}\,\mathrm{d} x \nonumber \\
& \geq \frac{s}{m+2}(\mu_m-\lambda)\geq \frac{s}{m+2}\Delta. \label{eq:log-diff}
\end{align}
Substituting \eqref{eq:detailed-root-log-comparison} and \eqref{eq:log-diff} into
\eqref{eq:detailed-eta-ratio}, we obtain
$\log\frac{\eta}{\eta_m}\geq\log c + \frac{s}{m+2}\Delta - O(m^{-1})$.
This proves \eqref{eq:first-complete-comparison}.

We now derive the sharper asymptotic expansions.  Assume $\Delta=O(1)$
and put $A=s/m$.
For $x>2$, we obtain
\begin{equation}\label{eq:detailed-log-P}
\log p_s(x) = (s+1)\log\sigma(x) - \log\big(\sigma(x) - \sigma(x)^{-1}\big)
+ \log\left(1 - \sigma(x)^{-2s-2}\right).
\end{equation}
For $x=m+O(1)$, the last term is $O(\sigma(x)^{-2s-2}) = o(1)$.
Moreover, $\sigma(x) - \sigma(x)^{-1} = \sqrt{x^2-4}$ and
$(\log\sigma(x))' = (x^2 - 4)^{-1/2}$.
Differentiating \eqref{eq:detailed-log-P} and using $s = \Theta(m\log m)$, we obtain, uniformly for
$x=m+O(1)$,
\begin{align}
\frac{\mathrm{d}}{\mathrm{d} x}\log p_s(x)
& = \frac{s+1}{\sqrt{x^2-4}} - \frac{x}{x^2-4} + o(1) \nonumber \\
& = \frac{s}{m} + O\Big(\frac{s}{m^2}\Big) + O(m^{-1}) + o(1) \nonumber \\
& = \frac{s}{m} + o(1) = A + o(1), \label{eq:derivative-log}
\end{align}
Now $\lambda=m-1-\Delta$ and $\mu_m=m-1+O(m^{-2})$, so $\mu_m-\lambda=\Delta+O(m^{-2})$.
Integrating \eqref{eq:derivative-log} over this interval
gives $\log p_s(\mu_m) - \log p_s(\lambda)
= (A + o(1))(\mu_m - \lambda) = A\Delta + o(1)$.
Combining this with $\log (z_u/z_u^{(m)}) = \log c + o(1)$
in \eqref{eq:detailed-eta-ratio} yields $\log (\eta/\eta_m)
= \log c + A\Delta + o(1)$.

It remains to prove the last inequality. Let
$\mu_{m-1} = \lambda_1(X_{s+1}(K_{m-1},u))$
and let $a_{m-1}$ be the coordinate at $u$ in its positive unit
Perron vector.  Lemma~\ref{lem:complete-block-estimates}, applied with
$m-1$ in place of $m$, gives $\mu_{m-1} = m - 2 + O(m^{-2})$ and
$a_{m-1} = (m-1)^{-1/2} (1+O(m^{-2}))$. Then
$\eta_{m-1} = \frac{a_{m-1}}{p_{s+1}(\mu_{m-1})}$. Therefore
\begin{equation}\label{eq:detailed-eta-mminusone-ratio}
\log\frac{\eta}{\eta_{m-1}}
= \log\frac{z_u}{a_{m-1}} +
\log p_{s+1}(\mu_{m-1}) -\log p_s(\lambda).
\end{equation}
Since $a_m/a_{m-1} = 1 + o(1)$, Equation \eqref{eq:detailed-root-log-comparison} gives
\begin{equation}\label{eq:temporary-3}
\log\frac{z_u}{a_{m-1}}
= \log\frac{z_u}{a_m} + \log\frac{a_m}{a_{m-1}} = \log c+o(1).
\end{equation}
To compare the two path polynomials, split the difference as
\begin{align}
\log p_{s+1}(\mu_{m-1}) - \log p_s(\lambda)
& = \left(\log p_{s+1}(\mu_{m-1}) - \log p_s(\mu_{m-1})\right) \nonumber \\
& \quad + \left(\log p_s(\mu_{m-1}) - \log p_s(\lambda)\right).
\label{eq:detailed-P-splitting}
\end{align}

The first bracket accounts for the additional edge of the pendant
path.  From the explicit formula for $p_s$,
\[
\begin{aligned}
 \frac{p_{s+1}(x)}{p_s(x)}
& = \sigma(x) \frac{1 - \sigma(x)^{-2s-4}}{1 - \sigma(x)^{-2s-2}}.
\end{aligned}
\]
For $x=m+O(1)$, the fraction on the right is $1+o(1)$, while
$\sigma(x)=m(1+o(1))$. Therefore
\begin{equation}\label{eq:detailed-one-more-edge}
\log p_{s+1}(x) - \log p_s(x)
= \log m+o(1) \qquad (x=m+O(1)).
\end{equation}
Applying this at $x=\mu_{m-1}$ shows that the first bracket in
\eqref{eq:detailed-P-splitting} equals
$\log m+o(1)$. For the second bracket, observe that
\[
\mu_{m-1} - \lambda =
\big(m - 2 + O(m^{-2})\big) - (m - 1 - \Delta)
= \Delta - 1 + O(m^{-2}).
\]
Using \eqref{eq:derivative-log} gives
$\log p_s(\mu_{m-1}) - \log p_s(\lambda) = A(\Delta - 1) + o(1)$.
Substituting this and \eqref{eq:detailed-one-more-edge} into
\eqref{eq:detailed-P-splitting}, we obtain
$\log p_{s+1}(\mu_{m-1}) - \log p_s(\lambda) = \log m+A(\Delta-1)+o(1)$.
Substituting this and \eqref{eq:temporary-3} into \eqref{eq:detailed-eta-mminusone-ratio} yields
the desired result.
\end{proof}

\begin{lemma}\label{lem:endpoint-interval-calculation}
Assume that $1\leq d\leq D$, $\lambda=m+O(1)$, and
$A\geq (1 - \frac1{2D})\log m$.
For $c$ in the range $(1 - O(m^{-1})) \lambda^{-d}, 1)$, define
\[
 \mathcal E(c)= \log c+A\Psi_d(c)+\max\{0,\log m-A\}.
\]
Then $\mathcal E(c)>0$ for all sufficiently large $m$, except possibly in a
shrinking neighborhood of $c=1$. 
\end{lemma}

\begin{proof}
We first determine lower bounds for $\Psi_d(c)$. We consider the following three cases.
\medskip

\noindent\emph{Case 1}: $\lambda^{-k}\leq c\leq\lambda^{-(k-1)}$ for some $2\leq k\leq d$.
For $1\leq i\leq k-1$, we see $\lambda^i c \leq\lambda^{i-k+1}\leq 1$. Hence
$\min\{1, \lambda^ic\} = \lambda^ic$ when $2\leq i\leq k-1$.
For $k\leq i\leq d$, then $\lambda^i c\geq\lambda^{i-k}\geq 1$. Therefore
$\min\{1, \lambda^ic\} = 1$ when $k\leq i\leq d$. It follows that
\begin{align}
\Psi_d(c)
& = d + 2 - c - \lambda\, c - \sum_{i=2}^{k-1} \lambda^ic - (d-k+1) \nonumber \\
& = k + 1 - (1 + \lambda + \cdots + \lambda^{k-1}) c.  \label{eq:Psi-Ik-detailed}
\end{align}
We now substitute the bound into $\mathcal{E}(c)$. For $2\leq k\leq d$, define
\[
f_k(c)= \log c+ A\big[k+1-(1+\lambda+\cdots+\lambda^{k-1})c\big] + \max\{0,\log m-A\}.
\]
Note that $\mathcal{E}(c)\geq f_k(c)$. Moreover,
$f_k''(c) = -1/c^2 < 0$.
Thus $f_k$ is strictly concave in $c$. Hence, we have
$f_k(c)\geq \min\{f_k(\lambda^{-k}), f_k(\lambda^{-(k-1)})\}$.
At the left endpoint $c=\lambda^{-k}$,
$(1+\lambda+\cdots+\lambda^{k-1})\lambda^{-k} = O(\lambda^{-1})$.
Therefore
\begin{align}
f_k(\lambda^{-k})
& = -k\log\lambda +(k+1)A + \max\{0,\log m-A\} + O(A/\lambda) \nonumber \\
& = -k\log m + (k+1)A + \max\{0,\log m-A\} +o(1). \label{eq:f_k-lower-bound-1}
\end{align}
At the right endpoint $c=\lambda^{-(k-1)}$, we have
$(1+\lambda+\cdots+\lambda^{k-1}) \lambda^{-(k-1)} = 1 + O(\lambda^{-1})$. Hence
\begin{align}
f_k(\lambda^{-(k-1)})
& = -(k-1)\log\lambda + kA + \max\{0,\log m-A\} + O(A/\lambda) \nonumber \\
& = -(k-1)\log m + kA + \max\{0,\log m-A\} +o(1). \label{eq:f_k-lower-bound-2}
\end{align}
We show that both endpoint values are positive. This is clear 
for $A\geq \log m$. Suppose next that $A<\log m$.
The principal term in \eqref{eq:f_k-lower-bound-1} becomes
\begin{align*}
 -k\log m + (k+1)A + \max\{0,\log m-A\}
 & = -k\log m + (k+1)A + \log m - A \\
 & = \log m\Big[1 - k\Big(1 - \frac{A}{\log m}\Big)\Big] > \frac{\log m}{2},
\end{align*}
where the last inequality uses the hypothesis
$1 - \frac{A}{\log m}\leq\frac1{2D}$ and
$k\leq d\leq D$.
Similarly, the principal term in \eqref{eq:f_k-lower-bound-2} becomes
$-(k-1)\log m + kA + \max\{0,\log m-A\} > (\log m)/2$.
By concavity, $\mathcal{E}(c) > 0$.

By Lemma \ref{lem:exact-defect-identity}, $c\geq (1 - O(m^{-1}))\lambda^{-d}$,
which may be slightly smaller than $\lambda^{-d}$.  Replacing
$\lambda^{-d}$ by $(1-O(m^{-1}))\lambda^{-d}$
changes $\log c$ by only $O(m^{-1})$.  Each truncated term
$\min\{1,\lambda^ic\}$ changes by at most
$\lambda^iO(m^{-1})\lambda^{-d}=O(m^{-1})$; after multiplication by
$A=O(\log m)$, the total change is $o(1)$ because $d\leq D$ is fixed.
Therefore the preceding positivity
remains valid on the entire allowed interval below $\lambda^{-1}$.
\medskip

\noindent\emph{Case 2}: $\lambda^{-1}\leq c\leq 2\lambda^{-1}$.
Then $1\leq\lambda\, c\leq 2$, so $\min\{2,\lambda c\} = \lambda c$.
For every $i\geq2$, $\lambda^ic\geq\lambda^{i-1}\geq 1$,
and hence $\min\{1,\lambda^ic\}=1$. Substituting into
\eqref{eq:Psi-defect-lower} gives
\begin{equation}\label{eq:Psi-Ione-low-detailed}
\Psi_d(c) = d+2-c-\lambda c-(d-1) = 3 - (\lambda + 1)c.
\end{equation}
We next consider $\lambda^{-1}\leq c\leq2\lambda^{-1}$. Define
$g(c) = \log c + A[3 - (\lambda + 1)c] + \max\{0,\log m-A\}$. Again,
$g''(c) = -c^{-2} < 0$, so it suffices to check the endpoints.
At $c=\lambda^{-1}$,
\begin{equation}\label{eq:g-lower-bound-1}
g(\lambda^{-1}) = -\log m + 2A + \max\{0,\log m-A\} + o(1).
\end{equation}
If $A\geq \log m$, this is at least $\log m + o(1)$.  If $A < \log m$, 
then the right-hand side of \eqref{eq:g-lower-bound-1} becomes $A + o(1) > 0$. 
Thus $g(\lambda^{-1})>0$. At $c=2\lambda^{-1}$,
\begin{equation}
g(2\lambda^{-1}) = \log2+(A - \log m)+ \max\{0, \log m - A\} + o(1).
\end{equation}
If $A < \log m$, the right-hand side is $\log 2 + o(1) > 0$.
If $A\geq \log m$, then it is at least
$\log2+o(1)>0$. Hence, $\mathcal{E}(c) > 0$.
\medskip 

\noindent\emph{Case 3}: $2\lambda^{-1}\leq c\leq 1$.
Then $\lambda\, c\geq 2$, so $\min\{2,\lambda c\}=2$.
Also, for $i\geq 2$, $\lambda^ic\geq 2\lambda^{i-1}>1$,
so $\min\{1,\lambda^ic\}=1$. Consequently,
$\Psi_d(c)=d+2-c-2-(d-1)=1-c$. Let $h(c):=\log c + A(1-c) + \max\{0, \log m - A\}$.
Again, $h''(c)=-c^{-2}<0$. At the left endpoint, we have
$h(2\lambda^{-1}) > 0$. At the right endpoint, $h(1) = \max\{0, \log m - A\}\geq 0$.
If $\max\{0, \log m - A\} > 0$, concavity immediately gives $h(c) > 0$ throughout the
whole interval $2\lambda^{-1}\leq c\leq 1$. If $\max\{0, \log m - A\} = 0$, 
then $h(1) = 0$, but strict concavity, together with $h(2\lambda^{-1})>0$, gives
$h(c)>0$. However, the positive lower bound tends to zero as $c\to1$.
Therefore this argument does not give a uniform positive margin in a
neighborhood of $c=1$. 
\end{proof}

\begin{lemma}\label{lem:near-complete-terminal}
Assume that $u$ is not a Perron maximum on $F$ and that $F\neq K_m$.
Let $c\in (1-o(1), 1)$. Then $\log\frac{\eta}{\eta_m}>0$.
\end{lemma}

\begin{proof}
Suppose that $F\neq K_m$ and $c\in (1-o(1),1)$. Choose two nonadjacent vertices $a,b\in V(F)$, and put $e=ab$.
Then $F\subseteq K_m-e$. Therefore $X_s(F,u)\subseteq X_s(K_m-e, u)$.
Hence, $\lambda\leq\lambda_1(X_s(K_m - e, u))$.
Lemma~\ref{lem:complete-block-estimates} gives
$\mu_m - \lambda_1(X_s(K_m-e, u)) \geq\frac{\kappa}{m}$.
Consequently, $\mu_m - \lambda\geq\frac{\kappa}{m}$.
Put $x=1-c=o(1)$. Since $0<c\leq1$, we have $x\geq0$. Also, $\Delta\geq 1 - c = x$.
Suppose first that $x\geq\frac{\kappa}{4m}$.
By \eqref{eq:first-complete-comparison} and $\Delta\geq x$, we obtain
\[
\log\frac{\eta}{\eta_m}
\geq\log c+\frac{s}{m+2}\Delta-O(m^{-1})
\geq\log(1-x)+\frac{s}{m+2}x-O(m^{-1}).
\]
For $0\leq x\leq 1/2$, $\log(1-x)\geq -x-x^2$. Hence
\[
\log\frac{\eta}{\eta_m}\geq
\left(\frac{s}{m+2}-1\right)x-x^2-O(m^{-1}) > 0,
\]
where the last inequality uses
$\frac{s}{m+2}=(1+o(1))\log m$.

Suppose next that $0\leq x<\frac{\kappa}{4m}$.
It follows from $\mu_m - \lambda\geq\frac{\kappa}{m}$ and \eqref{eq:log-diff} that
\[
\log\frac{\eta}{\eta_m}\geq\log c + \frac{\kappa s}{m(m+2)} - O(m^{-1}).
\]
Since $c=1-x$ and $x<\kappa/(4m)$, we have $\log c=\log(1-x)\geq-2x\geq-\frac{\kappa}{2m}$.
Hence,
\[
\log\frac{\eta}{\eta_m} > \frac{\kappa s}{m(m+2)} - O(m^{-1})
= \Omega\Big(\frac{\log m}{m}\Big).
\]
This finishes the proof.
\end{proof}

\begin{proposition}\label{prop:canonical-terminal-replacement}
At least one of the following holds:
\begin{enumerate}
\item[$(1)$] $F = K_m$;
\item[$(2)$] $u$ is a leaf of $F$, and, if $v$ is the unique neighbor
 of $u$ in $F$, then $X_s(F,u) = X_{s+1}(F-u,v)$;
\item[$(3)$] replacing $(F,s)$ by either $(K_m,s)$ or $(K_{m-1},s+1)$ strictly decreases the free-end Perron
 coordinate and the ratio of these two coordinates is at least $\mathrm{exp}\left(\Omega\left(\frac{\log m}{m}\right)\right)$.
\end{enumerate}
In the third alternative, the logarithm of the ratio between the old
and new free-end coordinates is bounded below by a positive constant
outside the near-complete range. In the near-complete range it is at
least $\Omega((\log m)/m)$.
\end{proposition}

\begin{proof}
Choose a vertex $w\in V(F)$ at which $\bm{z}$ is maximum on $F$, and let
$u = v_0,v_1,\ldots,v_d = w$ be a shortest $u$\,--\,$w$ path in $F$. 
If $F = K_m$, then (1) holds, so there is nothing to prove. Henceforth suppose that
$F\neq K_m$.

We divide the proof into the following two cases.
\medskip

\noindent\emph{Case 1}: $z_u < z_w$. In this case, $d\geq 1$, also by assumption, $1\leq d\leq D$.
\smallskip

\emph{Subcase 1.1}: $\deg_F(u) = 1$.
Let $v$ be the unique neighbor of $u$ in $F$. Since $u$ is a
leaf and $F$ is connected, the graph $F-u$ is connected. Hence, $X_s(F,u) = X_{s+1}(F-u, v)$.
This is (2).
\smallskip

\emph{Subcase 1.2}: $\deg_F(u)\geq 2$.
Since $z_u < z_w$, Lemma~\ref{lem:exact-defect-identity} and Lemma \ref{lem:endpoint-interval-calculation} gives
$\Delta\geq 0$. Since $\lambda = m + O(1)$, we also have
$\Delta = m - 1 -\lambda = O(1)$. 
Define $A:=s/m$. Because $s/(m\log m)\to 1$, we have
$A/\log m\to 1$. It follows that, for all sufficiently large
$m$, $A\geq (1 - 1/(2D)) \log m$. 

By Lemma~\ref{lem:two-complete-competitors},
\begin{align}
\max\left\{\log\frac{\eta}{\eta_m}, \log\frac{\eta}{\eta_{m-1}}\right\} 
& = \log c+A\Delta+\max\{0,\log m-A\}+o(1) \\
& \geq \log c+A\Psi_d(c)+\max\{0,\log m-A\}+o(1).
 \label{eq:canonical-max-comparison}
\end{align}
Lemma \ref{lem:endpoint-interval-calculation} shows that $\mathcal E(c)$ has a
positive margin whenever $c$ is outside a shrinking interval next
to $1$. Consequently, outside that interval,
$\max\big\{\log\frac{\eta}{\eta_m}, \log\frac{\eta}{\eta_{m-1}}\big\} > 0$.
Therefore either $\eta_m<\eta$ or $\eta_{m-1}<\eta$.
When $c$ is very close to $1$, strictness follows from
Lemma~\ref{lem:near-complete-terminal}.

This proves (3) whenever $u$ is not a Perron maximum and $\deg_F(u)\geq 2$.
\medskip

\noindent\emph{Case 2:} $z_u = z_w$.

The Perron-vector recurrence along the pendant path gives
$z_{y_j}\leq 3\lambda^{-j} z_u$ ($1\leq j\leq s$).
Since $\lambda = m + O(1)$, it follows that
\begin{align}
\sum_{j=1}^s z_{y_j}^2 \leq 9z_u^2\sum_{j=1}^s \lambda^{-2j}
\leq 9z_u^2\frac{\lambda^{-2}}{1-\lambda^{-2}} = O(m^{-2}) z_u^2.
 \label{eq:canonical-handle-mass}
\end{align}
Since $\|\bm{z}\|_2 = 1$, we have
\[
1 = \sum_{v\in V(F)} z_v^2 + \sum_{j=1}^s z_{y_j}^2\leq mz_u^2 + O(m^{-2}) z_u^2.
\]
Therefore $z_u\geq m^{-1/2} \big(1-O(m^{-3})\big)$.
Lemma~\ref{lem:complete-block-estimates} gives
$a_m = m^{-1/2} (1 + O(m^{-2}))$. Hence,
$\frac{z_u}{a_m}\geq 1 - O(m^{-2})$. Taking logarithms gives
\begin{equation}
 \log\frac{z_u}{a_m}\geq -O(m^{-2}).
 \label{eq:canonical-root-log-bound}
\end{equation}

Recall that we are assuming $F\neq K_m$. Choose a missing edge
$e$ of $F$, we have $\mu_m - \lambda\geq\frac{\kappa}{m}$.
Using $\eta = z_u/p_s(\lambda)$ and $\eta_m = a_m/p_s(\mu_m)$, we have
$\log\frac{\eta}{\eta_m} = \log\frac{z_u}{a_m} + \log p_s(\mu_m) - \log p_s(\lambda)$.
By \eqref{eq:canonical-root-log-bound}, the first term is at least
$-O(m^{-2})$.  For the second term,
\eqref{eq:detailed-log-derivative} and
\eqref{eq:missing-edge-pole-gain} give
\[
 \log p_s(\mu_m) - \log p_s(\lambda) \geq
 \frac{s}{m+2}(\mu_m-\lambda)\geq
 \frac{\kappa s}{m(m+2)}.
\]
It follows that
\[
\log\frac{\eta}{\eta_m}\geq -O(m^{-2})
+ \frac{\kappa s}{m(m+2)} = \Omega\Big(\frac{\log m}{m}\Big) > 0.
\]
Thus $\eta_m<\eta$. Replacing $F$ by $K_m$, while keeping the same handle length,
strictly decreases the free-end coordinate. This is (3).
\end{proof}

\section{Completion of the proof}

We now combine the structural results from the preceding sections.  Recall
that the two terminal graphs are rooted at the vertices incident with the
joining path.  If such a root has degree one in its terminal graph, we may
regard that root as the first vertex of the path instead.  We call this
operation \emph{absorbing a leaf into the path}; it changes only the
description of the graph, not the graph itself.

\begin{lemma}\label{lem:Bi-clique}
The induced subgraph $G[V(B_i)] = B_i$ is complete graph.
\end{lemma}

\begin{proof}
For $i=1,2$, write $m_i:=|V(B_i)|$, $\alpha_i=\lambda_1(B_i)$,
and let $a_i$ be the coordinate at $u_i$ of the unit Perron vector of
$B_i$.  For the 
function $F_i$ defined before
Lemma~\ref{lem:loaded-corridor-poles}, let $\vartheta_i$ be its unique zero
near $q$, and put $D_i = F_i'(\vartheta_i)$.

For a path with $s$ internal vertices, the two largest eigenvalues of $G$
are the two relevant solutions of
\begin{equation}
 F_1(x)F_2(x)-\frac{1}{p_s(x)^2}=0.
 \label{eq:completion-eigenvalue-equation}
\end{equation}
Each $F_i$ is strictly increasing near its zero.  If, for example,
$\vartheta_1\geq\vartheta_2$, then
$F_1(x)F_2(x)\leq0$ for
$\vartheta_2\leq x\leq\vartheta_1$.  Thus
\eqref{eq:completion-eigenvalue-equation} has no solution in that
interval.  The larger solution lies above both $\vartheta_i$, and the
smaller solution lies below both.  It follows that
\begin{equation}
 |\vartheta_1-\vartheta_2|\leq\gap(G)\leq q^{-60}.
 \label{eq:terminal-root-closeness}
\end{equation}

Choose $x_*$ between $\vartheta_1$ and $\vartheta_2$, and define the 
term
\begin{equation}
 T_{12}:= \frac{1}{p_s(x_*)\sqrt{D_1D_2}}.
 \label{eq:path-interaction-mixed}
\end{equation}
By Lemma~\ref{lem:loaded-corridor-poles},
$D_i=a_i^{-2} (1+O(q^{-2}))$.
Moreover, the formula for $p_s$ in \eqref{eq:path-hyperbolic} gives
\[
\frac1{p_s(x_*)} = (1 - \sigma(x_*)^{-2}) \sigma(x_*)^{-s}
\left(1 - \sigma(x_*)^{-2s-2}\right)^{-1}.
\]
Here $x_*=q+O(1)$, so $\sigma(x_*) = \Theta(q)$; since
$s=\Theta(q\log q)$, the last factor is $1+o(q^{-2})$. Therefore
\begin{equation}\label{eq:path-interaction-estimate}
T_{12} = (1 - \sigma(x_*)^{-2}) \sigma(x_*)^{-s} a_1a_2 (1 + O(q^{-2})).
\end{equation}
Lemma~\ref{lem:loaded-corridor-poles}\,(3) and \eqref{eq:terminal-root-closeness} also imply
\begin{equation}\label{eq:gap-T12}
\gap(G)\geq 2T_{12} (1 - O(q^{-2})).
\end{equation} 
Moreover, the value of $x_*$ in \eqref{eq:path-interaction-estimate} may be replaced
by either $\vartheta_i$ without affecting the relative error.  Indeed,
since $x_*$ lies between $\vartheta_1$ and $\vartheta_2$, \eqref{eq:terminal-root-closeness} gives
$|x_* - \vartheta_i|\leq |\vartheta_1 - \vartheta_2|\leq q^{-60}$.
All these values are $q+O(1)$, so
$(\log\sigma)'(x) = (x^2-4)^{-1/2} = O(q^{-1})$.
By the mean value theorem,
$|\log\sigma(x_*) - \log\sigma(\vartheta_i)| = O(q^{-1})\,O(q^{-60}) = O(q^{-61})$.
Since $s=\Theta(q\log q)$, we have
$s|\log\sigma(x_*) - \log\sigma(\vartheta_i)| = O(q^{-60}\log q) = o(q^{-2})$.
Consequently,
\[
\frac{\sigma(x_*)^{-s}}{\sigma(\vartheta_i)^{-s}}
= \exp\left(-s\big[\log\sigma(x_*) - \log\sigma(\vartheta_i)\big]\right)
= 1 + o(q^{-2}).
\]
Thus replacing $x_*$ by either $\vartheta_i$ introduces only a $1+o(q^{-2})$ multiplicative factor, which is absorbed by the existing $1+O(q^{-2})$ relative error.
\medskip

Put $d=m_2-m_1=O(1)$. Let $G_1$ consist of two copies of
$B_1$ joined by a path with $s+d$ internal vertices, and let
$G_2$ consist of two copies of $B_2$ joined by a path with $s-d$
internal vertices. Both graphs have exactly $m_1+m_2+s=|V(G)|$
vertices. As before, let $T_{11}$ and $T_{22}$ denote the analogues of
$T_{12}$ for these two graphs. Applying
\eqref{eq:path-interaction-estimate} to the three graphs gives
\[
\frac{T_{11}}{T_{12}} = \frac{a_1}{a_2}\sigma(x_*)^{-d} \big(1 + O(q^{-2})\big), \quad
\frac{T_{22}}{T_{12}} = \frac{a_2}{a_1}\sigma(x_*)^d \big(1 + O(q^{-2})\big).
\]
Consequently,
\begin{equation}
\frac{T_{11}T_{22}}{T_{12}^2} = 1 + O(q^{-2}).
\label{eq:completion-reflection-product}
\end{equation}
In particular, at least one of $T_{11}, T_{22}$ at most $(1 + O(q^{-2})) T_{12}$.  

We may therefore divide $G_i$ into two isomorphic halves at the midpoint of its joining path.
Depending on the parity of the path length, the two halves either are
joined by one edge or are both adjacent to one vertex. In either
case a half has the form $X_L(B_i,u_i)$ for some
$L = (s+O(1))/2 = (1+o(1))m_i\log m_i$.

The spectral radius of this half lies between $m_i-C$ and $m_i$ for an
absolute constant $C$.  Its Perron vector decreases along the pendant
path, so a largest coordinate occurs in $B_i$.  By
Theorem~\ref{thm:terminal-extraction}, every vertex of $B_i$ is at distance
at most $11$ from $u_i$.  Hence the hypotheses of
Proposition~\ref{prop:canonical-terminal-replacement} hold for each half.

First absorb every leaf root into the path. This does not alter either
the graph or the Perron coordinate at the free end of the half.  There are
at most $11$ such absorptions. If the remaining terminal graph is a
clique, leave it unchanged and set $c_i=1$. Otherwise,
by Proposition~\ref{prop:canonical-terminal-replacement}, replaces the half,
without changing its number of vertices, by either
$X_L(K_m,u)$ or $X_{L+1}(K_{m-1},u)$. If $\eta_i$ and
$\eta_i^{\mathrm{cl}}$ are the unit Perron-vector coordinates at the free
end before and after this replacement, respectively, then
\begin{equation}\label{eq:completion-clique-replacement-gain}
\eta_i^{\mathrm{cl}} = c_i\eta_i, \qquad
c_i\leq\exp\Big\{-\kappa\frac{\log q}{q}\Big\}
\end{equation}
for some absolute constant $\kappa>0$. Make the same replacement in both halves of $G_i$, and call the
resulting graph $G_i^{\mathrm{cl}}$. If a terminal loses
one vertex, the path gains one vertex on that side, so the parity of the
midpoint is unchanged.

Suppose first that the midpoint of $G_i$ is an edge. The two halves are
then joined by an edge between vertices whose Perron coordinates are both
$\eta_i$. Lemma~\ref{lem:uniform-two-pole} gives
\begin{equation}\label{eq:completion-edge-replacement-ratio}
\frac{\gap(G_i^{\mathrm{cl}})}{\gap(G_i)}
= c_i^2 \big(1 + O(q^{-2})\big).
\end{equation}
If the midpoint is a vertex, Lemma~\ref{lem:zero-parity} gives the same
comparison, with the slightly larger error
\begin{equation}\label{eq:completion-vertex-replacement-ratio}
\frac{\gap(G_i^{\mathrm{cl}})}{\gap(G_i)} = c_i^2 \big(1 + O(q^{-1})\big).
\end{equation}
Indeed, the additional denominator in that lemma is the spectral radius of
a half; both before and after replacement it is $q+O(1)$.

For $i\in\{1,2\}$, define $\mathcal{R}_i := \frac{\gap(G_i^{\mathrm{cl}})}{2T_{12}}$.
Lemma \ref{lem:loaded-corridor-poles} (3) show that $\gap(G_i) = 2T_{ii}(1 + O(q^{-1}))$. It follows from
\eqref{eq:completion-edge-replacement-ratio}--\eqref{eq:completion-vertex-replacement-ratio} that
\[
\mathcal{R}_i \leq\frac{T_{ii}}{T_{12}}c_i^2 \big(1 + O(q^{-1})\big).
\]
Multiplying these inequalities and using \eqref{eq:completion-reflection-product},
we obtain $\mathcal{R}_1\mathcal{R}_2\leq(c_1c_2)^2 (1 + O(q^{-1}))$.
If at least one reduced terminal graph is not complete, then at least one
of $c_1,c_2$ satisfies \eqref{eq:completion-clique-replacement-gain}, while the other is
at most $1$. Hence
\[
\mathcal{R}_1\mathcal{R}_2
= 1 - \Omega\Big(\frac{\log q}{q}\Big).
\]
Therefore, for some $i\in\{1,2\}$, we have $\gap(G_i^{\mathrm{cl}})
\leq2T_{12} (1 - \Omega((\log q)/q))$.
Since $(\log q)/q\gg q^{-2}$, this inequality and
\eqref{eq:gap-T12} give $\gap(G_i^{\mathrm{cl}}) < \gap(G)$,
a contradiction finishing the proof
\end{proof}
 
Now, we are ready to finfish the proof of Theorem~\ref{thm:main}.
\smallskip

\begin{proof}[Proof of Theorem~\ref{thm:main}]
Let $G$ be a connected $n$-vertex graph with minimum spectral gap. For
sufficiently large $n$, Theorem \ref{thm:terminal-extraction} shows that 
$G$ is obtained from the disjoint union $B_1\mathbin{\dot\cup}C\mathbin{\dot\cup} B_2$
by adding precisely the two edges $\widehat{u}_1\widehat{v}_1$ and $\widehat{u}_2\widehat{v}_2$,
where $\widehat{v}_1$, $\widehat{v}_2\in V(C)$ and $\widehat{u}_i\in V(B_i)$, $i=1,2$,
each $B_i$ have $q + O(1)$ vertices, where $q\sim n/(2\log n)$. Proposition~\ref{prop:corridor-is-path}
shows that $C$ is a path with $n - 2q + O(1) = (2+o(1))q\log q$ vertices.
After absorbing leaf roots into
the path whenever possible, Lemma \ref{lem:Bi-clique} shows that both $B_i$ are
cliques.

It remains to show that these two cliques have the same order. Let their
orders be $q_1$ and $q_2$. For a clique $K_r$, Lemma~\ref{lem:loaded-corridor-poles} (2) gives
$\vartheta(K_r) = r - 1 + O(q^{-2})$.
If $q_1\neq q_2$, then
\[
|\vartheta (K_{q_1}) - \vartheta(K_{q_2})|
= |q_1 - q_2| + O(q^{-2}) \geq 1 - O(q^{-2}),
\]
contradicting \eqref{eq:terminal-root-closeness}. Therefore $q_1=q_2$.
Therefore every connected $n$-vertex graph of minimum spectral gap is a
double kite, as claimed.
\end{proof}

\section*{Acknowledgments}
    
This work began while the second author was visiting the School of Mathematical Sciences, Anhui University in June 2025. The authors thank Josh Tobin for sharing his overall ideas on Stani\'c's conjecture and its relation with his work with the second author on the principal ratio of a graph.
In the early technical stage,
the authors used language-model-based tools to brainstorm candidate proof strategies for certain lemmas based on our high-level ideas.
Any suggestions arising from these tools served only as informal inspiration. The high-level strategy of the proof was formulated by the authors before any artificial intelligence tools were used in any way.

\end{document}